\documentclass[12pt]{amsart}

\usepackage{amsmath}
\usepackage{amssymb}
\usepackage{amsthm}
\usepackage{pstricks}
\usepackage{pst-node}

\usepackage{verbatim}

\numberwithin{equation}{section}

\def\sgn{\operatorname{sgn}}
\def\wt{\operatorname{wt}}

\newtheorem{thm}{Theorem}[section]
\newtheorem{prop}[thm]{Proposition}
\newtheorem{lem}[thm]{Lemma}
\newtheorem{cor}[thm]{Corollary}

\theoremstyle{definition}
\newtheorem{rem}[thm]{Remark}
\newtheorem{defin}[thm]{Definition}
\newtheorem{ex}[thm]{Example}

\begin{document}
\title{The Pentagram Map and $Y$-patterns}
\author{Max Glick}
\address{Department of Mathematics, University of Michigan,
Ann Arbor, MI 48109, USA} \email{maxglick@umich.edu}


\begin{abstract}
The pentagram map, introduced by R. Schwartz, is defined by the following construction: given a polygon as input, draw all of its ``shortest'' diagonals, and output the smaller polygon which they cut out.  We employ the machinery of cluster algebras to obtain explicit formulas for the iterates of the pentagram map.
\end{abstract}

\date{\today}
\thanks{Partially supported by NSF grants DMS-0943832 and DMS-0555880.} 

\maketitle

\section{Introduction and main formula}
The pentagram map, introduced by Richard Schwartz, is a geometric construction which produces one polygon from another.  Figure \ref{figureT} gives an example of this operation.  Schwartz \cite{S} uses a collection of cross ratio coordinates to study various properties of the pentagram map.  In this paper, we work with a related set of quantities, which we term the $y$-parameters.  A polygon can be reconstructed (up to a projective transformation) from its $y$-parameters together with a few other quantities.  The other quantities transform in a very simple manner under the pentagram map, so a good understanding of the map can be gained by determining how it affects the $y$-parameters.  

\begin{figure}[ht] \label{figureT}
\begin{pspicture}(13,5)
\rput(0,-.5){
	\pspolygon[linewidth=2pt](1,2)(1,3)(2,4.5)(3,5)(4.5,4.5)(5,3.5)(4.5,2)(3.5,1)(2,1)
  \psline[linewidth=2pt]{->}(5.5,3)(7.5,3)

  \rput(7,0){
  \pspolygon(1,2)(1,3)(2,4.5)(3,5)(4.5,4.5)(5,3.5)(4.5,2)(3.5,1)(2,1)
  \pspolygon[linestyle=dashed](1,2)(2,4.5)(4.5,4.5)(4.5,2)(2,1)(1,3)(3,5)(5,3.5)(3.5,1)
  \pspolygon[linewidth=2pt](1.22,2.55)(1.67,3.67)(2.5,4.5)(3.67,4.5)(4.5,3.87)(4.5,2.67)(3.97,1.79)(2.75,1.3)(1.62,1.75)
  }
}
\end{pspicture}
\caption{The pentagram map}
\end{figure}
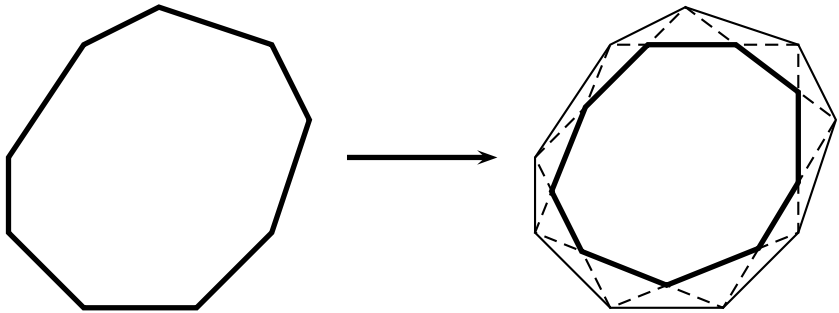

It turns out that the $y$-parameters of a polygon output by the pentagram map can be computed as simple rational functions of the $y$-parameters of the original polygon.  Moreover, these rational functions are precisely the transition equations of the $Y$-pattern associated to a certain cluster algebra.  We exploit this connection to derive formulas for the $y$-parameters of a polygon obtained by applying the pentagram map multiple times.  These formulas depend on the $F$-polynomials of the corresponding cluster algebra, which in general are defined recursively.  In this instance, a non-recursive description of these polynomials can be found.  Specifically, the $F$-polynomials are generating functions for the order ideals of a certain sequence of partially ordered sets.  These posets were originally defined by N. Elkies, G. Kuperberg, M. Larsen, and J. Propp  \cite{EKLP}.  It is clear from this description of the $F$-polynomials that they have positive coefficients, verifying that the Laurent positivity conjecture of S. Fomin and A. Zelevinsky \cite{FZ2} holds in this case.

This paper is organized as follows.  In the remainder of this section we state our main result, the formula for the $y$-parameters of the iterated pentagram map.  This formula is proven in the subsequent sections.  Section~ 2 gives the transition equations of the $y$-parameters under a single application of the pentagram map.  In Section~ 3, we explain the connection to $Y$-patterns.  This connection is used in Section~ 4 to derive our main formula in terms of the $F$-polynomials.  Section~ 4 also provides an analogous formula expressed in the original coordinate system used by Schwartz.  In Section 5 we present background on alternating sign matrices and related concepts including order ideals and the octahedron recurrence.  Section 6 contains the proof of the formula for the $F$-polynomials in terms of these order ideals.  Lastly, Section~ 7 applies the results of this paper to axis-aligned polygons, expanding on a result of Schwartz.

Schwartz \cite{S} studies the pentagram map on a class of objects called twisted polygons.  A \emph{twisted polygon} is a sequence $A=(A_i)_{i\in \mathbb{Z}}$ of points in the projective plane that is periodic modulo some projective transformation $\phi$, i.e., $A_{i+n} = \phi(A_i)$ for all $i \in \mathbb{Z}$.  Two twisted polygons $A$ and $B$ are said to be \emph{projectively equivalent} if there exists a projective transformation $\psi$ such that $\psi(A_i) = B_i$ for all $i$.  Let $\mathcal{P}_n$ denote the space of twisted $n$-gons modulo projective equivalence.  

The \emph{pentagram map}, denoted $T$, inputs a twisted polygon $A$ and constructs a new twisted polygon $T(A)$ given by the following sequence of points:
\begin{displaymath}
\ldots, \overleftrightarrow{A_{-1}A_1} \cap \overleftrightarrow{A_0A_2}, \overleftrightarrow{A_0A_2} \cap \overleftrightarrow{A_1A_3}, \overleftrightarrow{A_{1}A_3} \cap \overleftrightarrow{A_2A_4}, \ldots
\end{displaymath}
(we denote by $\overleftrightarrow{AB}$ the line passing through $A$ and $B$.)  Note that this operation is only defined for generic twisted polygons.  Specifically, the lines $\overleftrightarrow{A_{i-1}A_{i+1}}$ and $\overleftrightarrow{A_iA_{i+2}}$ must be distinct for all $i$ in order for the pentagram map to be applied.  The pentagram map preserves projective equivalence, so it is well defined for generic points of $\mathcal{P}_n$.  

A complication arises when trying to index the sequence $T(A)$.  It would be equally reasonable to assign the point $\overleftrightarrow{A_{i-1}A_{i+1}} \cap \overleftrightarrow{A_iA_{i+2}}$ the index $i$ or $i+1$.  Instead of doing either, we use $\frac{1}{2} + \mathbb{Z} = \{\ldots, -0.5,0.5,1.5,2.5\ldots\}$ to label $B = T(A)$.  Specifically, we let
\begin{displaymath}
B_i = \overleftrightarrow{A_{i-\frac{3}{2}}A_{i+\frac{1}{2}}} \cap \overleftrightarrow{A_{i-\frac{1}{2}}A_{i+\frac{3}{2}}}
\end{displaymath}
for all $i \in (\frac{1}{2} + \mathbb{Z})$.  This indexing scheme is illustrated in Figure \ref{figureTindexed}.  Similarly, if $B$ is a sequence of points indexed by $\frac{1}{2} + \mathbb{Z}$ then $T(B)$ is defined in the same way and is indexed by $\mathbb{Z}$.  Let $\mathcal{P}_n^*$ denote the space of twisted $n$-gons indexed by $\frac{1}{2} + \mathbb{Z}$, modulo projective equivalence.

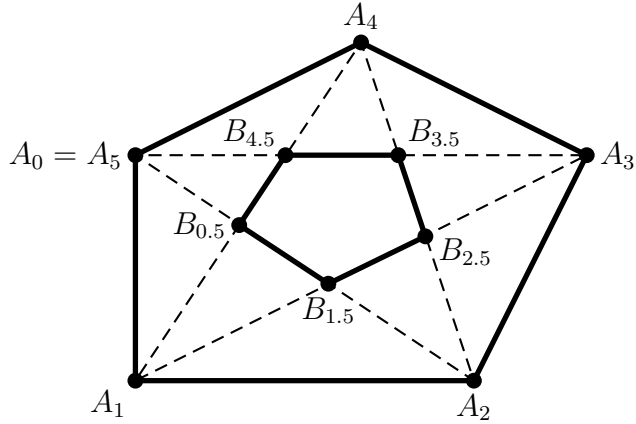
\begin{figure} 
\psset{unit=1.5cm}
\begin{pspicture}(6,4)
	\rput(0,-0.5){
  \pspolygon[showpoints=true,linewidth=2pt](1,1)(4,1)(5,3)(3,4)(1,3)
  \uput[225](1,1){$A_1$}
  \uput[270](4,1){$A_2$}
  \uput[0](5,3){$A_3$}
  \uput[90](3,4){$A_4$}
  \uput[l](1,3){$A_0=A_5$}
  \pspolygon[linestyle=dashed](1,1)(5,3)(1,3)(4,1)(3,4)
  \pspolygon[showpoints=true,linewidth=2pt](1.92,2.38)(2.71,1.86)(3.57,2.28)(3.33,3)(2.33,3)
  \uput[d](2.71,1.86){$B_{1.5}$}
  \uput[330](3.57,2.28){$B_{2.5}$}
  \uput[ur](3.33,3){$B_{3.5}$}
  \uput[ul](2.33,3){$B_{4.5}$}
  \uput[l](1.92,2.38){$B_{0.5}$}
  }
\end{pspicture}
\psset{unit=1cm}
\caption{The pentagon $B = T(A)$ is indexed by $\frac{1}{2} + \mathbb{Z}$.}
\label{figureTindexed}
\end{figure}

The \emph{cross ratio} of 4 real numbers $a,b,c,d$ is defined to be 
\begin{displaymath}
\chi(a,b,c,d) = \frac{(a-b)(c-d)}{(a-c)(b-d)}.
\end{displaymath}
This definition extends to the projective line, on which it gives a projective invariant of 4 points.  We will be interested in taking the cross ratio of 4 collinear points in the projective plane, or dually, the cross ratio of 4 lines intersecting at a common point.

\begin{defin}
Let $A$ be a twisted polygon indexed either by $\mathbb{Z}$ or $\frac{1}{2} + \mathbb{Z}$.  The \emph{$y$-parameters} of $A$ are the real numbers $y_j(A)$ for $j \in \mathbb{Z}$ defined as follows.  For each index $k$ of $A$ let
\begin{align} 
y_{2k}(A) = -\left(\chi (\overleftrightarrow{A_{k}A_{k-2}}, \overleftrightarrow{A_{k}A_{k-1}}, \overleftrightarrow{A_{k}A_{k+1}}, \overleftrightarrow{A_kA_{k+2}})\right)^{-1} \label{yvertex} \\
y_{2k+1}(A)= -\chi (\overleftrightarrow{A_{k-2}A_{k-1}} \cap L, A_k, A_{k+1}, \overleftrightarrow{A_{k+2}A_{k+3}} \cap L) \label{yedge} 
\end{align}
where $L = \overleftrightarrow{A_kA_{k+1}}$.
\end{defin}

Note that the 4 lines in \eqref{yvertex} all pass through the point $A_k$, and the 4 points in \eqref{yedge} all lie on the line $L$.  Therefore the cross ratios are defined.  These cross ratios are illustrated in Figure \ref{figureyA}.

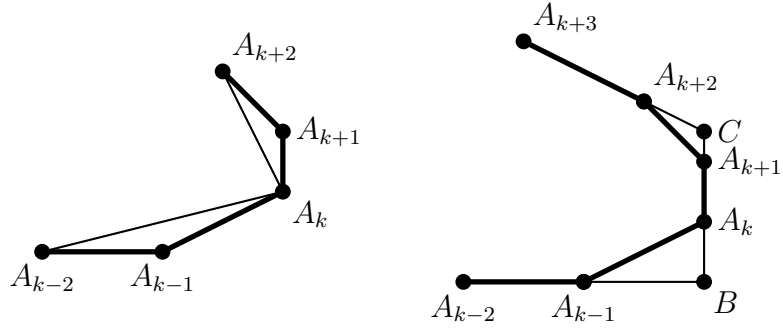
\begin{figure}
\psset{unit=.8cm}
\begin{pspicture}(12,6)
\rput(0,.5){
\psline[showpoints=true,linewidth=2pt](1,1)(3,1)(5,2)(5,3)(4,4)
\uput[d](1,1){$A_{k-2}$}
\uput[d](3,1){$A_{k-1}$}
\uput[dr](5,2){$A_k$}
\uput[r](5,3){$A_{k+1}$}
\uput[ur](4,4){$A_{k+2}$}
\psline(1,1)(5,2)(4,4)
}
\rput(7,0){
\psline[showpoints=true,linewidth=2pt](1,1)(3,1)(5,2)(5,3)(4,4)(2,5)
\uput[d](1,1){$A_{k-2}$}
\uput[d](3,1){$A_{k-1}$}
\uput[r](5,2){$A_k$}
\uput[r](5,3){$A_{k+1}$}
\uput[ur](4,4){$A_{k+2}$}
\uput[ur](2,5){$A_{k+3}$}
\psline[showpoints=true,dotsize=2pt 5](3,1)(5,1)(5,3.5)(4,4)
\uput[dr](5,1){$B$}
\uput[r](5,3.5){$C$}
}
\end{pspicture}
\psset{unit=1cm}
\caption{The cross ratios corresponding to the $y$-parameters.  On the left, $-(y_{2k}(A))^{-1}$ is the cross ratio of the 4 lines through $A_k$.  On the right, $y_{2k+1}(A)= -\chi(B,A_k,A_{k+1}, C)$.} \label{figureyA}
\end{figure}

As will be demonstrated, each $y$-parameter of $T(A)$ can be expressed as a rational function of the $y$-parameters of $A$.  It follows that each iterate of $T$ corresponds to a rational map of the $y$-parameters.  Our formulas for these maps involve the $F$-polynomials of a particular cluster algebra.  These can in turn be expressed in terms of certain posets which we define now.

The original definition of the posets, given by Elkies, Kuperberg, Larsen, and Propp \cite{EKLP}, involves height functions of domino tilings.  Although we will use this characterization later, the following self-contained definition suffices for now.  Let $Q_k$ be the set of triples $(r,s,t)\in \mathbb{Z}^3$ such that
\begin{displaymath}
2|s|-(k-2) \leq t \leq  k-2-2|r|
\end{displaymath}
and
\begin{displaymath}
2|s|-(k-2) \equiv t \equiv  k-2-2|r| \pmod{4}.
\end{displaymath}
Let $P_k = Q_{k+1} \cup Q_k$.  Define a partial order on $P_k$ by saying that $(r',s',t')$ covers $(r,s,t)$  if and only if $t' = t + 1$ and $|r'-r| + |s'-s| = 1$.  The Hasse diagrams of $P_2$ and $P_3$ are given in Figure \ref{figureP2}.  

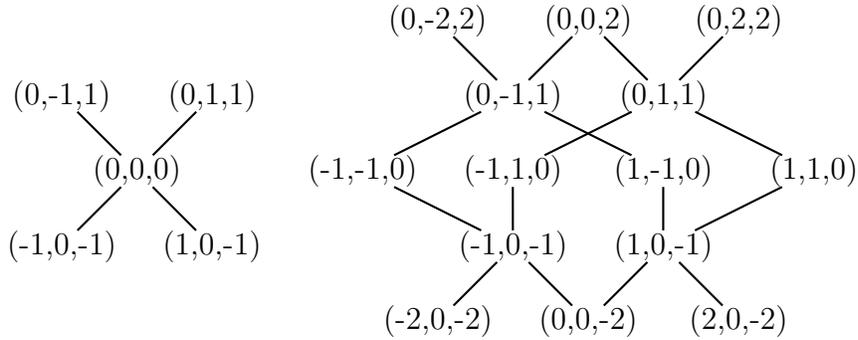
\begin{figure}
\begin{pspicture}(11,6)
\rput(1,4){\rnode{v1}{(0,-1,1)}}
\rput(3,4){\rnode{v2}{(0,1,1)}}
\rput(2,3){\rnode{v3}{(0,0,0)}}
\rput(1,2){\rnode{v4}{(-1,0,-1)}}
\rput(3,2){\rnode{v5}{(1,0,-1)}}
\ncline{v4}{v3}
\ncline{v5}{v3}
\ncline{v3}{v1}
\ncline{v3}{v2}

\rput(4,0){
\rput(2,5){\rnode{w1}{(0,-2,2)}}
\rput(4,5){\rnode{w2}{(0,0,2)}}
\rput(6,5){\rnode{w3}{(0,2,2)}}
\rput(3,4){\rnode{w4}{(0,-1,1)}}
\rput(5,4){\rnode{w5}{(0,1,1)}}
\rput(1,3){\rnode{w6}{(-1,-1,0)}}
\rput(3,3){\rnode{w7}{(-1,1,0)}}
\rput(5,3){\rnode{w8}{(1,-1,0)}}
\rput(7,3){\rnode{w9}{(1,1,0)}}
\rput(3,2){\rnode{w10}{(-1,0,-1)}}
\rput(5,2){\rnode{w11}{(1,0,-1)}}
\rput(2,1){\rnode{w12}{(-2,0,-2)}}
\rput(4,1){\rnode{w13}{(0,0,-2)}}
\rput(6,1){\rnode{w14}{(2,0,-2)}}

\ncline{w12}{w10}
\ncline{w13}{w10}
\ncline{w13}{w11}
\ncline{w14}{w11}
\ncline{w10}{w6}
\ncline{w10}{w7}
\ncline{w11}{w8}
\ncline{w11}{w9}
\ncline{w6}{w4}
\ncline{w7}{w5}
\ncline{w8}{w4}
\ncline{w9}{w5}
\ncline{w4}{w1}
\ncline{w4}{w2}
\ncline{w5}{w2}
\ncline{w5}{w3}
}
\end{pspicture}
\caption{The poset $P_k$ for $k=2$ (left) and $k=3$ (right)}
\label{figureP2}
\end{figure}

We denote by $J(P_k)$ the set of \emph{order ideals} in $P_k$, i.e., subsets $I \subseteq P_k$ such that $x \in I$ and $y < x$ implies $y \in I$.

\begin{thm} \label{thmTk}
Let $A \in \mathcal{P}_n$ and let $y_j = y_j(A)$ for all $j \in \mathbb{Z}$.  If $k \geq 1$ then the $y$-parameters of $T^k(A)$ are given by 
\begin{equation} \label{Tk}
y_j(T^k(A)) = \begin{cases}
                 \left(\displaystyle\prod_{i=-k}^k y_{j+3i}\right)\dfrac{F_{j-1,k}F_{j+1,k}}{F_{j-3,k}F_{j+3,k}}, & j+k \textrm{ even} \\
                 \left(\displaystyle\prod_{i=-k+1}^{k-1} y_{j+3i}\right)^{-1} \dfrac{F_{j-3,k-1}F_{j+3,k-1}}{F_{j-1,k-1}F_{j+1,k-1}}, & j+k \textrm{ odd} \\
              \end{cases}
\end{equation} 
where 
\begin{equation} \label{Fjk}
F_{j,k} = \sum_{I \in J(P_k)} \prod_{(r,s,t)\in I} y_{3r+s+j}.
\end{equation}
\end{thm}

\begin{ex}
Take $k=2$ in Theorem \ref{thmTk}.  Now $P_1=\{(0,0,0)\}$ has 2 order ideals, namely $\emptyset$ and $P_1$ itself.  So by \eqref{Fjk} we have $F_{j,1} = 1 + y_j$.  Then \eqref{Tk} becomes
\begin{displaymath}
y_j(T^2(A)) = \frac{1}{y_{j-3}y_jy_{j+3}}\frac{(1+y_{j-3})(1+y_{j+3})}{(1+y_{j-1})(1+y_{j+1})}
\end{displaymath}
for $j$ odd.  
Meanwhile, Figure \ref{figureP2weights} shows a copy of $P_2$ in which each vertex $(r,s,t)$ has been labeled with the corresponding variable $y_{3r+s+j}$.  This poset has eight order ideals.  The four which do not contain $y_j$ contribute terms $1$, $y_{j-3}$, $y_{j+3}$, and $y_{j-3}y_{j+3}$ to $F_{j,2}$, adding up to $(1+y_{j-3})(1+y_{j+3})$.  The other four order ideals contribute terms which sum to $y_{j-3}y_jy_{j+3}(1+y_{j-1})(1+y_{j+1})$.  Adding these yields 
\begin{displaymath}
F_{j,2} = (1+y_{j-3})(1+y_{j+3}) + y_{j-3}y_jy_{j+3}(1+y_{j-1})(1+y_{j+1}),
\end{displaymath}
and in terms of these polynomials, \eqref{Tk} gives
\begin{displaymath}
y_j(T^2(A)) = y_{j-6}y_{j-3}y_jy_{j+3}y_{j+6}\frac{F_{j-1,2}F_{j+1,2}}{F_{j-3,2}F_{j+3,2}}
\end{displaymath}
for $j$ even.
\end{ex}

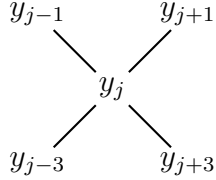
\begin{figure}
\begin{pspicture}(5,4)
\rput(1,3){\rnode{v1}{$y_{j-1}$}}
\rput(3,3){\rnode{v2}{$y_{j+1}$}}
\rput(2,2){\rnode{v3}{$y_j$}}
\rput(1,1){\rnode{v4}{$y_{j-3}$}}
\rput(3,1){\rnode{v5}{$y_{j+3}$}}
\psset{nodesep=2pt}
\ncline{v4}{v3}
\ncline{v5}{v3}
\ncline{v3}{v1}
\ncline{v3}{v2}
\psset{nodesep=0pt}
\end{pspicture}
\caption{The poset $P_2$ with each element $(r,s,t)$ labeled by $y_{3r+s+j}$}
\label{figureP2weights}
\end{figure}

Throughout this paper, we adopt the convention that $\prod_{i=a}^{a-1}z_i = 1$ and $\prod_{i=a}^{b}z_i = \prod_{i=b+1}^{a-1}(1/z_i)$ for $b < a-1$.  This will frequently allow a single formula to encompass what otherwise would require several cases.  With this convention, the property $\prod_{i=a}^bz_i \prod_{i=b+1}^cz_i = \prod_{i=a}^cz_i$ holds for all $a,b,c \in \mathbb{Z}$.

\medskip

\textbf{Acknowledgments.} I thank Sergey Fomin for suggesting this problem and providing valuable guidance throughout.

\section{The transition equations} 
Let $A$ be a twisted $n$-gon.  Since the cross ratio is invariant under projective transformations, it follows that $y_{j+2n}(A) = y_j(A)$ for all $j$.  In this section, we show that each $y$-parameter of $T(A)$ is a rational function of $y_1(A),\ldots,y_{2n}(A)$.  The proof of this fact makes use of the cross ratio coordinates $x_1,\ldots,x_{2n}$ introduced by Schwartz \cite{S}.  

For each index $k$ of $A$ let
\begin{align*}
x_{2k}(A) &= \chi (A_{k-2},A_{k-1},\overleftrightarrow{A_kA_{k+1}} \cap \overleftrightarrow{A_{k-2}A_{k-1}}, \overleftrightarrow{A_{k+1}A_{k+2}} \cap \overleftrightarrow{A_{k-2}A_{k-1}}) \\
x_{2k+1}(A) &= \chi(A_{k+2},A_{k+1},\overleftrightarrow{A_{k}A_{k-1}} \cap \overleftrightarrow{A_{k+2}A_{k+1}}, \overleftrightarrow{A_{k-1}A_{k-2}} \cap \overleftrightarrow{A_{k+2}A_{k+1}}). 
\end{align*}
This definition makes sense as all 4 points in the first cross ratio lie on the line $\overleftrightarrow{A_{k-2}A_{k-1}}$ and those in the second all lie on the line $\overleftrightarrow{A_{k+2}A_{k+1}}$ (see Figure \ref{figurexA}).  As with the $y_j$, we have that the $x_j$ are periodic mod $2n$.  

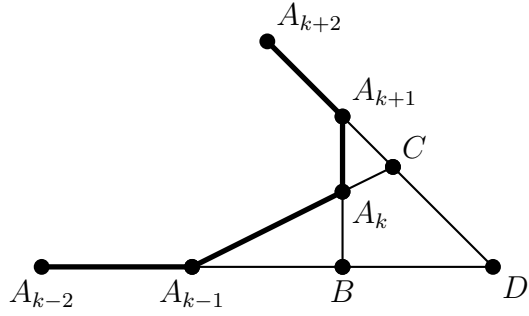
\begin{figure}
\begin{pspicture}(10,6)
\rput(2,0){
\psline[showpoints=true,linewidth=2pt](1,1)(3,1)(5,2)(5,3)(4,4)
\uput[d](1,1){$A_{k-2}$}
\uput[d](3,1){$A_{k-1}$}
\uput[dr](5,2){$A_k$}
\uput[ur](5,3){$A_{k+1}$}
\uput[ur](4,4){$A_{k+2}$}
\psline[showpoints=true,dotsize=2pt 5](3,1)(5,1)(5,2)(5.67,2.33)(5,3)
\psline[showpoints=true,dotsize=2pt 5](5,1)(7,1)(5.67,2.33)
\uput[d](5,1){$B$}
\uput[ur](5.67,2.33){$C$}
\uput[dr](7,1){$D$}
}
\end{pspicture}
\caption{The $x$-coordinates of $A$.  Here, $x_{2k}(A) = \chi(A_{k-2},A_{k-1},B,D)$ and $x_{2k+1}(A) = \chi(A_{k+2},A_{k+1},C,D)$.} \label{figurexA}
\end{figure}

\begin{prop}[\cite{S}] The functions $x_1,\ldots,x_{2n}$ are (generically) a set of coordinates of the space $\mathcal{P}_n$ and of the space $\mathcal{P}_n^*$.
\end{prop}

\begin{prop}[{\cite[(7)]{S}}] \label{propalpha12x} Suppose that $(x_1,\ldots,x_{2n})$ are the $x$-coordinates of $A$.  If $A$ is indexed by $\frac{1}{2} + \mathbb{Z}$ then $T(A)$ has $x$-coordinates $(x_1',\ldots,x_{2n}')$ where
\begin{displaymath}
x_j' = \begin{cases} 
                  x_{j-1}\frac{1-x_{j-3}x_{j-2}}{1-x_{j+1}x_{j+2}},
                  &j \textrm{ even} \\ 
                  x_{j+1}\frac{1-x_{j+3}x_{j+2}}{1-x_{j-1}x_{j-2}},
                  &j \textrm{ odd} \\ 
              \end{cases} 
\end{displaymath}
Alternately, if $A$ is indexed by $\mathbb{Z}$ then the $x$-coordinates of $T(A)$ are
\begin{displaymath}
x_j'' = \begin{cases} 
                  x_{j+1}\frac{1-x_{j+3}x_{j+2}}{1-x_{j-1}x_{j-2}},
                  &j \textrm{ even} \\ 
                  x_{j-1}\frac{1-x_{j-3}x_{j-2}}{1-x_{j+1}x_{j+2}},
                  &j \textrm{ odd} \\ 
             \end{cases}
\end{displaymath}
\end{prop}

As observed by V. Ovsienko, R. Schwartz, and S. Tabachnikov in \cite{OST}, the products $x_jx_{j+1}$ are themselves cross ratios.  In fact, $x_jx_{j+1}$ equals the cross ratios used in \eqref{yvertex}--\eqref{yedge} to define $y_j$.  Therefore
\begin{equation} \label{yvertexx}
y_j = -(x_jx_{j+1})^{-1}
\end{equation}
if $j/2$ is an index of $A$ and 
\begin{equation} \label{yedgex}
y_j = -(x_jx_{j+1})
\end{equation}
otherwise. 

It follows that $y_1y_2y_3\cdots y_{2n}=1$ for any twisted polygon.  Therefore the $y$-parameters do not ``coordinatize'' $\mathcal{P}_n$ or $\mathcal{P}_n^*$.  However, the $y_j$ together with Schwartz's pentagram invariants $O_k$ and $E_k$ \cite{S} can be used to determine the $x$-coordinates and hence the polygon (up to projective transformation.)  More precisely, the $y$-parameters determine the products $x_jx_{j+1}$ for $j=1,\ldots,2n$.  These, together with $E_n = x_2x_4\cdots x_{2n}$, can be used to compute $x_j^n$ for each $j$.  If $n$ is odd then this is all that is needed to find the $x$-coordinates.  On the other hand, if $n$ is even then $y_1,\ldots, y_{2n},$ and $E_n$ only determine the $x$-coordinates up to a simultaneous change of sign.  In this event, another pentagram invariant such as $E_1$ can be used to resolve the ambiguity.

The pentagram invariants are interchanged by the pentagram map: $O_k(T(A)) = E_k(A)$ and $E_k(T(A)) = O_k(A)$ for all twisted polygons $A$ (Theorem~ 1.1 of \cite{S}).  What remains, then, is to understand how the pentagram map and its iterates affect the $y$-parameters.

\begin{prop} \label{propalpha12}
Let $(y_1,\ldots,y_{2n})$ be the $y$-parameters of $A$.  If $A$ is indexed by $\frac{1}{2} + \mathbb{Z}$ then $y_j(T(A)) = y_j'$ where
\begin{equation} 
\label{alpha1}
y_j' = \begin{cases} 
              y_{j-3}y_jy_{j+3}\frac{(1+y_{j-1})(1+y_{j+1})}{(1+y_{j-3})(1+y_{j+3})},
              &j \textrm{ even} \\ 
              y_j^{-1},
              &j \textrm{ odd} \\ 
        \end{cases} 
\end{equation}
If $A$ is indexed by $\mathbb{Z}$ then $y_j(T(A)) = y_j''$ where
\begin{equation}
\label{alpha2}        
y_j'' = \begin{cases} 
              y_j^{-1},          
              &j \textrm{ even} \\ 
              y_{j-3}y_jy_{j+3}\frac{(1+y_{j-1})(1+y_{j+1})}{(1+y_{j-3})(1+y_{j+3})},
              &j \textrm{ odd} \\ 
         \end{cases}        
\end{equation}
\end{prop}

\begin{proof}
We will prove the formula when $A$ is indexed by $\frac{1}{2} + \mathbb{Z}$.  If $j$ is odd then $j/2$ is an index of $A$ but not of $T(A)$, so $y_j = -(x_jx_{j+1})^{-1}$ and $y_j' = -x_j'x_{j+1}'$.  Computing
\begin{align*}
y_j' &= -x_j'x_{j+1}' \\
     &= -\left(x_{j+1}\frac{1-x_{j+3}x_{j+2}}{1-x_{j-1}x_{j-2}}\right)
         \left(x_{j}\frac{1-x_{j-2}x_{j-1}}{1-x_{j+2}x_{j+3}}\right) \\
     &= -x_jx_{j+1} \\
     &= y_j^{-1}. 
\end{align*}
On the other hand, if $j$ is even then $y_j = -(x_jx_{j+1})$ and 
\begin{align*}
y_j' &= -(x_j'x_{j+1}')^{-1} \\
     &= -\left(\left(x_{j-1}\frac{1-x_{j-3}x_{j-2}}{1-x_{j+1}x_{j+2}}\right)
         \left(x_{j+2}\frac{1-x_{j+4}x_{j+3}}{1-x_{j}x_{j-1}}\right)\right)^{-1} \\
     &= -\frac{x_jx_{j+1}}{x_{j-1}x_jx_{j+1}x_{j+2}}
        \frac{(1-x_{j-1}x_{j})(1-x_{j+1}x_{j+2})}{(1-x_{j-3}x_{j-2})(1-x_{j+3}x_{j+4})} \\
     &= y_{j-1}y_jy_{j+1}\frac{(1+1/y_{j-1})(1+1/y_{j+1})}{(1+1/y_{j-3})(1+1/y_{j+3})} \\
     &= y_{j-3}y_jy_{j+3}\frac{(1+y_{j-1})(1+y_{j+1})}{(1+y_{j-3})(1+y_{j+3})}
\end{align*}
as desired.  The case when $A$ is indexed by $\mathbb{Z}$ is similar.
\end{proof}

\begin{rem}
One can prove Proposition~ \ref{propalpha12} without using the $x$-coordinates at all.  By definition, the $y$-parameters are certain negative cross ratios.  It follows that the expressions $1+y_j$ and $1+1/y_j$ are also given by cross ratios.  The equations \eqref{alpha1}--\eqref{alpha2} then become multiplicative cross ratio identities which can be proven geometrically.
\end{rem}

\begin{rem}
Section 12.2 of the survey \cite{KNS} provides formulas analogous to \eqref{alpha1}--\eqref{alpha2} in the setting of quadrilateral lattices in 3-space.
\end{rem}
Let $\alpha_1$ be the rational map $(y_1,\ldots,y_{2n}) \mapsto (y'_1,\ldots,y'_{2n})$ defined by \eqref{alpha1}.  Similarly, let $\alpha_2$ be the rational map $(y_1,\ldots,y_{2n}) \mapsto (y''_1,\ldots,y''_{2n})$ defined by \eqref{alpha2}.  Proposition~ \ref{propalpha12} implies that the $y$-parameters transform under the map $T^k$ according to the rational map $\ldots \circ \alpha_1 \circ \alpha_2 \circ \alpha_1 \circ \alpha_2$ (the composition of $k$ functions), assuming the initial polygon is indexed by integers.  

\section{The associated $Y$-pattern}

The equations \eqref{alpha1}--\eqref{alpha2} can be viewed as transition equations of a certain $Y$-pattern.  $Y$-patterns represent a part of cluster algebra dynamics; they were introduced by Fomin and Zelevinsky \cite{FZ}.  A simplified (but sufficient for our current purposes) version of the relevant definitions is given below.

\begin{defin}
A $Y$-seed is a pair $(\mathbf{y},B)$ where $\mathbf{y}=(y_1,\ldots,y_n)$ is an $n$-tuple of quantities and $B$ is an $n \times n$ skew-symmetric, integer matrix.  The integer $n$ is called the \emph{rank} of the seed.  Given a $Y$-seed $(\mathbf{y},B)$ and some $k=1,\ldots,n$, the \emph{seed mutation} $\mu_k$ in direction $k$ results in a new $Y$-seed $\mu_k(\mathbf{y},B) = (\mathbf{y'},B')$ where
\begin{displaymath}
y_j' = \begin{cases} 
              y_j^{-1}, & j=k \\
              y_jy_k^{[b_{kj}]_+}(1+y_k)^{-b_{kj}}, & j \neq k \\
       \end{cases}
\end{displaymath}
and $B'$ is the matrix with entries
\begin{displaymath}
b_{ij}' = \begin{cases} 
              -b_{ij}, & i=k \textrm{ or } j=k \\
              b_{ij} + \sgn(b_{ik})[b_{ik}b_{kj}]_+, & \textrm{otherwise} \\
       \end{cases}
\end{displaymath}
In these formulas, $[x]_+$ is shorthand for $\max(x,0)$.  
\end{defin}

The data of the exchange matrix $B$ can alternately be represented by a \emph{quiver}.  This is a directed graph on vertex set $\{1,\ldots,n\}$.  For each $i$ and $j$, there are $|b_{ij}|$ arcs connecting vertex $i$ and vertex $j$.  Each such arc is oriented from $i$ to $j$ if $b_{ij} > 0$ and from $j$ to $i$ if $b_{ij} < 0$.  In terms of quivers, the mutation $\mu_k$ consists of the three following steps
\begin{enumerate}
\item For every length 2 path $i \to k \to j$, add an arc from $i$ to $j$.
\item Reverse the orientation of all arcs incident to $k$.
\item Remove all oriented 2-cycles.
\end{enumerate}

Figure \ref{figurequiver} illustrates some quiver mutations applied to the quiver associated with the exchange matrix
\begin{displaymath}
\left[
\begin{array}{cccccc}
0  &  1&  0& -1&  0&  0 \\
-1 &  0& -1&  0&  1&  0 \\
0  &  1&  0&  0&  0& -1 \\
1  &  0&  0&  0& -1&  0 \\
0  & -1&  0&  1&  0&  1 \\
0  &  0&  1&  0& -1&  0 \\
\end{array}
\right].
\end{displaymath}
Note that in this example the mutated quiver is the same as the initial one except that all the arrows have been reversed.  The is an instance of a more general phenomenon described by the following lemma.

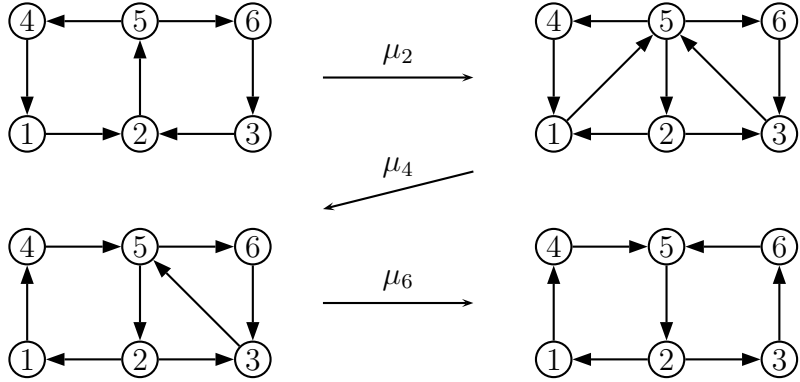
\begin{figure}
\begin{pspicture}(12,8)
\SpecialCoor
\rput(0,3){
\cnode(1,1){.25}{v10}
\rput(v10){1}
\cnode(2.5,1){.25}{v20}
\rput(v20){2}
\cnode(4,1){.25}{v30}
\rput(v30){3}
\cnode(1,2.5){.25}{v40}
\rput(v40){4}
\cnode(2.5,2.5){.25}{v50}
\rput(v50){5}
\cnode(4,2.5){.25}{v60}
\rput(v60){6}
\psset{arrowsize=5pt}
\psset{arrowinset=0}
\ncline{->}{v10}{v20}
\ncline{->}{v30}{v20}
\ncline{->}{v40}{v10}
\ncline{->}{v20}{v50}
\ncline{->}{v60}{v30}
\ncline{->}{v50}{v40}
\ncline{->}{v50}{v60}
}
\rput(7,3){
\cnode(1,1){.25}{v11}
\rput(v11){1}
\cnode(2.5,1){.25}{v21}
\rput(v21){2}
\cnode(4,1){.25}{v31}
\rput(v31){3}
\cnode(1,2.5){.25}{v41}
\rput(v41){4}
\cnode(2.5,2.5){.25}{v51}
\rput(v51){5}
\cnode(4,2.5){.25}{v61}
\rput(v61){6}
\psset{arrowsize=5pt}
\psset{arrowinset=0}
\ncline{->}{v11}{v51}
\ncline{->}{v31}{v51}
\ncline{->}{v21}{v11}
\ncline{->}{v21}{v31}
\ncline{->}{v41}{v11}
\ncline{->}{v51}{v21}
\ncline{->}{v61}{v31}
\ncline{->}{v51}{v41}
\ncline{->}{v51}{v61}
}
\rput(0,0){
\cnode(1,1){.25}{v12}
\rput(v12){1}
\cnode(2.5,1){.25}{v22}
\rput(v22){2}
\cnode(4,1){.25}{v32}
\rput(v32){3}
\cnode(1,2.5){.25}{v42}
\rput(v42){4}
\cnode(2.5,2.5){.25}{v52}
\rput(v52){5}
\cnode(4,2.5){.25}{v62}
\rput(v62){6}
\psset{arrowsize=5pt}
\psset{arrowinset=0}
\ncline{->}{v32}{v52}
\ncline{->}{v22}{v12}
\ncline{->}{v22}{v32}
\ncline{->}{v12}{v42}
\ncline{->}{v52}{v22}
\ncline{->}{v62}{v32}
\ncline{->}{v42}{v52}
\ncline{->}{v52}{v62}
}
\rput(7,0){
\cnode(1,1){.25}{v13}
\rput(v13){1}
\cnode(2.5,1){.25}{v23}
\rput(v23){2}
\cnode(4,1){.25}{v33}
\rput(v33){3}
\cnode(1,2.5){.25}{v43}
\rput(v43){4}
\cnode(2.5,2.5){.25}{v53}
\rput(v53){5}
\cnode(4,2.5){.25}{v63}
\rput(v63){6}
\psset{arrowsize=5pt}
\psset{arrowinset=0}
\ncline{->}{v23}{v13}
\ncline{->}{v23}{v33}
\ncline{->}{v13}{v43}
\ncline{->}{v53}{v23}
\ncline{->}{v33}{v63}
\ncline{->}{v43}{v53}
\ncline{->}{v63}{v53}
}
\psset{arrowsize=1.5pt 2}
\psline{->}(5,4.75)(7,4.75)
\uput[u](6,4.75){$\mu_2$}
\psline{->}(7,3.5)(5,3)
\uput[u](6,3.25){$\mu_4$}
\psline{->}(5,1.75)(7,1.75)
\uput[u](6,1.75){$\mu_6$}
\end{pspicture}
\caption{Some quiver mutations} \label{figurequiver}
\end{figure}

\begin{lem} \label{lemBipartite}
Suppose that $(\mathbf{y},B)$ is a $Y$-seed of rank $2n$ such that $b_{ij}=0$ whenever $i,j$ have the same parity (so the associated quiver is bipartite).  Assume also that for all $i$ and $j$ the number of length 2 paths in the quiver from $i$ to $j$ equals the number of length 2 paths from $j$ to $i$.  Then the $\mu_i$ for $i$ odd pairwise commute as do the $\mu_i$ for $i$ even.  Moreover, $\mu_{2n-1} \circ \cdots \circ \mu_3 \circ \mu_1(y,B) = (\mathbf{y'}, -B)$ and $\mu_{2n} \circ \cdots \circ \mu_4 \circ \mu_2(\mathbf{y},B) = (\mathbf{y''}, -B)$ where
\begin{align}
\label{muodd}
y'_j &= \begin{cases}
                y_j\prod_ky_k^{[b_{kj}]_+}(1+y_k)^{-b_{kj}}, & j \textrm{ even} \\
                y_j^{-1}, &      j \textrm{ odd} \\
        \end{cases} \\
\label{mueven}
y''_j &= \begin{cases}
                y_j^{-1}, &      j \textrm{ even} \\
                y_j\prod_ky_k^{[b_{kj}]_+}(1+y_k)^{-b_{kj}},  & j \textrm{ odd} \\
         \end{cases} 
\end{align}
\end{lem}

The proof of this lemma is a simple calculation using the description of quiver mutations above.  Note that the term bipartite, as used in the statement of the lemma, simply means that each arc in the quiver connects an odd vertex and an even vertex.  No condition on the orientation of the arcs is placed.  A stronger notion would require that all arcs begin at an odd vertex and end at an even one.  The discussion of bipartite belts in \cite{FZ} uses the stronger condition.  As such, the results proven there do not apply to the current context.  We will, however, use much of the same notation.

Let $\mu_{\textrm{even}}$ be the compound mutation $\mu_{\textrm{even}} = \mu_{2n} \circ \ldots \circ \mu_4 \circ \mu_2$ and let $\mu_{\textrm{odd}} = \mu_{2n-1} \circ \ldots \circ \mu_3 \circ \mu_1$.  Equations \eqref{alpha1}--\eqref{alpha2} and \eqref{muodd}--\eqref{mueven} suggest that $\alpha_1$ and $\alpha_2$ are instances of $\mu_{\textrm{odd}}$ and $\mu_{\textrm{even}}$, respectively. Indeed, let $B_0$ be the matrix with entries
\begin{displaymath}
b_{ij}^0 = \begin{cases}
            (-1)^j,   & i-j \equiv \pm 1 \pmod{2n} \\
             (-1)^{j+1},    & i-j \equiv \pm 3 \pmod{2n} \\
             0,             & \textrm{otherwise} \\
         \end{cases}
\end{displaymath}
The corresponding quiver in the case $n=8$ is shown in Figure \ref{figureB0}. 
\begin{figure}
\begin{pspicture}(-5,-5)(5,5)
\SpecialCoor
\cnode(2;22.5){.25}{v1}
\rput(v1){1}
\cnode(4;45){.25}{v2}
\rput(v2){2}
\cnode(2;67.5){.25}{v3}
\rput(v3){3}
\cnode(4;90){.25}{v4}
\rput(v4){4}
\cnode(2;112.5){.25}{v5}
\rput(v5){5}
\cnode(4;135){.25}{v6}
\rput(v6){6}
\cnode(2;157.5){.25}{v7}
\rput(v7){7}
\cnode(4;180){.25}{v8}
\rput(v8){8}
\cnode(2;202.5){.25}{v9}
\rput(v9){9}
\cnode(4;225){.25}{v10}
\rput(v10){10}
\cnode(2;247.5){.25}{v11}
\rput(v11){11}
\cnode(4;270){.25}{v12}
\rput(v12){12}
\cnode(2;292.5){.25}{v13}
\rput(v13){13}
\cnode(4;315){.25}{v14}
\rput(v14){14}
\cnode(2;337.5){.25}{v15}
\rput(v15){15}
\cnode(4;0){.25}{v16}
\rput(v16){16}

\psset{arrowsize=5pt}
\psset{arrowinset=0}
\ncline{->}{v1}{v16}   \ncline{->}{v1}{v2}
\ncline{->}{v3}{v2}    \ncline{->}{v3}{v4} 
\ncline{->}{v5}{v4}    \ncline{->}{v5}{v6}
\ncline{->}{v7}{v6}    \ncline{->}{v7}{v8}
\ncline{->}{v9}{v8}    \ncline{->}{v9}{v10}
\ncline{->}{v11}{v10}  \ncline{->}{v11}{v12}
\ncline{->}{v13}{v12}  \ncline{->}{v13}{v14}
\ncline{->}{v15}{v14}  \ncline{->}{v15}{v16}
\ncline{->}{v16}{v13} \ncline{->}{v16}{v3}
\ncline{->}{v2}{v15} \ncline{->}{v2}{v5}
\ncline{->}{v4}{v1} \ncline{->}{v4}{v7}
\ncline{->}{v6}{v3} \ncline{->}{v6}{v9}
\ncline{->}{v8}{v5} \ncline{->}{v8}{v11}
\ncline{->}{v10}{v7} \ncline{->}{v10}{v13}
\ncline{->}{v12}{v9} \ncline{->}{v12}{v15}
\ncline{->}{v14}{v11} \ncline{->}{v14}{v1}

\end{pspicture}
\caption{The quiver associated with the exchange matrix $B_0$ for $n=8$} \label{figureB0}
\end{figure}
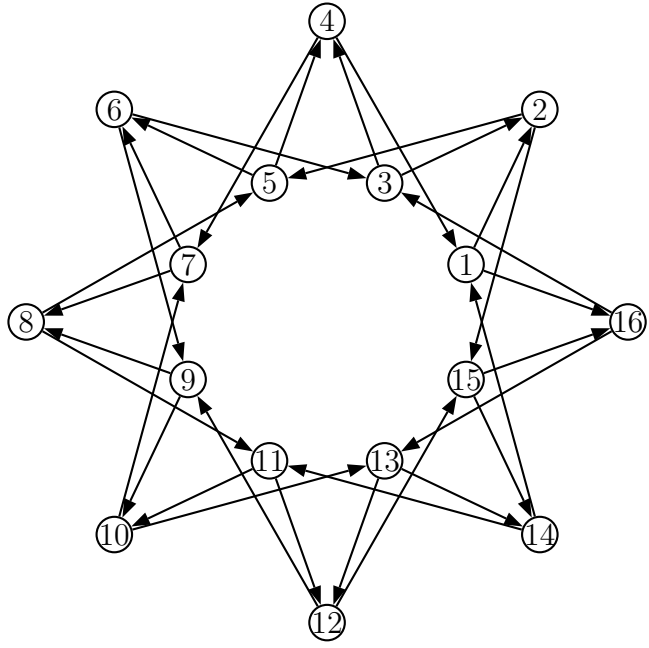

\begin{prop}
$\mu_{\textrm{even}}(\mathbf{y},B_0) = (\alpha_2(\mathbf{y}), -B_0)$ and $\mu_{\textrm{odd}}(\mathbf{y},-B_0) = (\alpha_1(\mathbf{y}), B_0)$.
\end{prop}

\begin{proof}
First of all, $B_0$ is skew-symmetric and $b_{i,j}^0 = 0$ for $i,j$ of equal parity.  In the quiver associated to $B_0$, the number of length 2 paths from $i$ to $j$ is 1 if $|i-j| \in \{2,4\}$ and 0 otherwise.  Therefore, Lemma~ \ref{lemBipartite} applies to $B_0$ and $\mu_{\textrm{even}}$ is given by \eqref{mueven}.

Both $\alpha_2$ and $\mu_{\textrm{even}}$ invert the $y_j$ for $j$ even.  Now suppose $j$ is odd.  Then $\alpha_2$ has the effect of multiplying $y_j$ by
\begin{displaymath}
y_{j-3}y_{j+3}\frac{(1+y_{j-1})(1+y_{j+1})}{(1+y_{j-3})(1+y_{j+3})}
\end{displaymath}
while $\mu_{\textrm{even}}$ multiplies $y_j$ by
\begin{displaymath}
\prod_ky_k^{[b_{kj}^0]_+}(1+y_k)^{-b_{kj}^0}.
\end{displaymath}
Since $j$ is odd, we have $b_{j\pm1,j}^0 = -1$ and $b_{j\pm3,j}^0 = 1$.  So these two factors agree.  This shows that $\alpha_2$ and $\mu_{\textrm{even}}$ have the same effect on the $y$-variables.  That $\mu_{\textrm{even}}$ negates the exchange matrix $B_0$ also follows from Lemma~ \ref{lemBipartite}.

The proof that $\alpha_1$ corresponds to the mutation $\mu_{\textrm{odd}}$, applied with exchange matrix $-B_0$, is similar.
\end{proof}

\section{The formula for an iterate of the pentagram map}
Let $A$ be a twisted $n$-gon indexed by $\mathbb{Z}$, and let $\mathbf{y}=(y_1,\ldots,y_{2n})$ be its $y$-parameters.  For $k \geq 0$ let $\mathbf{y}_k = (y_{1,k},\ldots,y_{2n,k})$ be the $y$-parameters of $T^k(A)$.  In other words, $\mathbf{y}_0 = \mathbf{y}$, $\mathbf{y}_{2m+1} = \alpha_2(\mathbf{y}_{2m})$, and $\mathbf{y}_{2m} = \alpha_1(\mathbf{y}_{2m-1})$.  The results of the previous section show that the $\mathbf{y}_k$ are related by seed mutations:
\begin{displaymath}
(\mathbf{y}_0,B_0) \xrightarrow{\mu_{\textrm{even}}} (\mathbf{y}_1,-B_0) \xrightarrow{\mu_{\textrm{odd}}} (\mathbf{y}_2,B_0) \xrightarrow{\mu_{\textrm{even}}} (\mathbf{y}_3,-B_0) \xrightarrow{\mu_{\textrm{odd}}} \cdots 
\end{displaymath}

Note that each $y_{j,k}$ is a rational function of $y_1,\ldots,y_{2n}$.  In the language of cluster algebras, this rational function is denoted $Y_{j,k} \in \mathbb{Q}(y_1,\ldots,y_{2n})$.  Explicitly, $Y_{j,0} = y_j$ and by \eqref{alpha1} and \eqref{alpha2}
\begin{displaymath}
Y_{j,k+1} = \begin{cases} 
              1/Y_{j,k},             &j+k \textrm{ even} \\ 
              Y_{j-3,k}Y_{j,k}Y_{j+3,k} \frac{(1+Y_{j-1,k})(1+Y_{j+1,k})} {(1+Y_{j-3,k})(1+Y_{j+3,k})},
              &j+k \textrm{ odd} \\ 
         \end{cases}
\end{displaymath}

To simplify formulas, it is easier to consider only the $Y_{j,k}$ for $j+k$ even.  The recurrence satisfied by these rational functions is $Y_{j,-1}= 1/y_j$ for $j$ odd, $Y_{j,0} = y_j$ for $j$ even, and
\begin{displaymath}
Y_{j,k} = \frac{Y_{j-3,k-1}Y_{j+3,k-1}}{Y_{j,k-2}} \frac{(1+Y_{j-1,k-1})(1+Y_{j+1,k-1})} {(1+Y_{j-3,k-1})(1+Y_{j+3,k-1})}
\end{displaymath}
for $j+k$ even and $k \geq 1$.  From these, it is easy to compute the other $Y_{j,k}$ because if $j+k$ is odd, then $Y_{j,k} = 1/Y_{j,k-1}$.

Proposition~ 3.13 of \cite{FZ}, specialized to the present context, says that if $j+k$ is even then $Y_{j,k}$ can be written in the form
\begin{equation} \label{Yjk}
Y_{j,k} = M_{j,k}\frac{F_{j-1,k}F_{j+1,k}}{F_{j-3,k}F_{j+3,k}}. 
\end{equation}
Here, $M_{j,k}$ is a Laurent monomial in $y_1,\ldots, y_{2n}$ and the $F_{i,k}$ are certain polynomials over $y_1,\ldots, y_{2n}$.  A description of these component pieces follows.

The monomial $M_{j,k}$ is given by the evaluation of the rational expressions $Y_{j,k}$ in the tropical semifield $\mathbb{P} = Trop(y_1,\ldots,y_{2n})$.  This is carried out as follows.  First of all, $Y_{j,k}$ is expressed in such a manner that no minus signs appear (that this is possible is clear from transition equations of the $Y$-pattern).  Next, each plus sign is replaced by the auxiliary addition $\oplus$ symbol.  This is a binary operation on Laurent monomials defined by $\prod_iy_i^{a_i} \oplus \prod_iy_i^{a_i'} = \prod_iy_i^{\min(a_i,a_i')}$.  Finally, this operation together with multiplication and division of monomials is used to compute a result.  As an example, 
\begin{displaymath}
Y_{3,1} = y_{0}y_3y_{6}\frac{(1+y_{2})(1+y_{4})}{(1+y_{0})(1+y_{6})}
\end{displaymath}
so
\begin{displaymath}
M_{3,1} = \left . y_{0}y_3y_{6}\frac{(1+y_{2})(1+y_{4})}{(1+y_{0})(1+y_{6})}\right |_\mathbb{P} 
                    = y_{0}y_3y_{6}\frac{(1\oplus y_{2})(1\oplus y_{4})}{(1\oplus y_{0})(1\oplus y_{6})} 
                    = y_{0}y_3y_{6}.
\end{displaymath}

Now $M_{j,-1} = Y_{j,-1} = 1/y_j$ for $j$ odd and $M_{j,0} = Y_{j,0} = y_j$ for $j$ even.  The transition equation for the monomials is identical to the transition equation for the $Y_{j,k}$, except that $+$ is replaced throughout by $\oplus$.  So, if $j+k$ is even and $k \geq 1$ then

\begin{displaymath}
M_{j,k} = \frac{M_{j-3,k-1}M_{j+3,k-1}}{M_{j,k-2}} \frac{(1\oplus M_{j-1,k-1})(1\oplus M_{j+1,k-1})} {(1\oplus M_{j-3,k-1})(1\oplus M_{j+3,k-1})}.
\end{displaymath}

\begin{prop} \label{propMjk}
The solution to this recurrence is given by 
\begin{equation} \label{Mjk}
M_{j,k} = \prod_{i=-k}^k y_{j+3i}
\end{equation}
for $j+k$ even.
\end{prop}

\begin{proof}
Clearly the initial conditions are satisfied.  Suppose $j+k$ is even and $k \geq 1$.  Then
\begin{align*}
M_{j,k} &= \prod_{i=-k}^k y_{j+3i} \\
          &= \frac{\prod_{i=-k+1}^{k-1} y_{j-3+3i}\prod_{i=-k+1}^{k-1} y_{j+3+3i}}{\prod_{i=-k+2}^{k-2} y_{j+3i}} \\ 
          &= \frac{M_{j-3,k-1}M_{j+3,k-1}}{M_{j,k-2}} \\
          &= \frac{M_{j-3,k-1}M_{j+3,k-1}}{M_{j,k-2}} \frac{(1\oplus M_{j-1,k-1})(1\oplus M_{j+1,k-1})} {(1\oplus M_{j-3,k-1})(1\oplus M_{j+3,k-1})}. 
\end{align*}
The last equality is justified because each $M_{j+i,k-1}$ for $i$ odd is an actual monomial (as opposed to a Laurent monomial), so $1 \oplus M_{j+i,k} = 1$ for these $i$.
\end{proof}

The $F_{j,k}$ for $j+k$ odd are defined recursively as follows.  Put $F_{j,-1}= 1$ for $j$ even, $F_{j,0} = 1$ for $j$ odd, and 
\begin{displaymath}
F_{j,k+1} = \frac{F_{j-3,k}F_{j+3,k} + M_{j,k}F_{j-1,k}F_{j+1,k}}{(1 \oplus M_{j,k})F_{j,k-1}} 
\end{displaymath}
for $j+k$ even and $k \geq 0$.  Recall $M_{j,k} = \prod_{i=-k}^k y_{j+3i}$ so this formula simplifies to
\begin{equation} \label{Fjkoddrec}
F_{j,k+1} = \frac{F_{j-3,k}F_{j+3,k} + (\prod_{i=-k}^k y_{j+3i})F_{j-1,k}F_{j+1,k}}{F_{j,k-1}}. 
\end{equation}
For example, $F_{j,1} = 1 + y_j$ and 
\begin{equation} \label{Fj2}
F_{j,2} = (1+y_{j-3})(1+y_{j+3}) + y_{j-3}y_jy_{j+3}(1+y_{j-1})(1+y_{j+1}).
\end{equation}
Although it is not clear from this definition, the $F_{j,k}$ are indeed polynomials.  This is a consequence of general cluster algebra theory.

Equations \eqref{Yjk}--\eqref{Mjk} and the fact that $Y_{j,k} = 1/Y_{j,k-1}$ for $j+k$ odd combine to prove that the formula given in Theorem~ \ref{thmTk} is of the right form.  What remains is to prove \eqref{Fjk}, which expresses the $F$-polynomials in terms of order ideals.  This proof is developed in the next several sections.  Before moving on we show how Theorem~ \ref{thmTk} can be used to derive a similar formula expressing the iterates of the pentagram map in the $x$-coordinates. 

\begin{thm}
Let $A \in \mathcal{P}_n$, $x_j = x_j(A)$, and $y_j = y_j(A)$.  Then $x_{j,k} = x_j(T^k(A))$ is given by 
\begin{equation} \label{Tkx}
x_{j,k} = \begin{cases}
             x_{j-3k}\left(\displaystyle\prod_{i=-k}^{k-1} y_{j+1+3i}\right)\dfrac{F_{j+2,k-1}F_{j-3,k}}{F_{j-2,k-1}F_{j+1,k}}, & j+k \textrm{ even} \\
             x_{j+3k}\left(\displaystyle\prod_{i=-k}^{k-1} y_{j+1+3i}\right)\dfrac{F_{j-3,k-1}F_{j+2,k}}{F_{j+1,k-1}F_{j-2,k}}, & j+k \textrm{ odd} \\
          \end{cases}
\end{equation}
\end{thm}

\begin{proof}
Let $y_{j,k} = y_j(T^k(A))$.  Based on the discussion in Section~ 2, the $x_{j,k}$ are uniquely determined by the identities
$x_{j,k}x_{j+1,k} = -y_{j,k}^{-1}$ for $j+k$ even, $x_{j,k}x_{j+1,k} = -y_{j,k}$ for $j+k$ odd, and
\begin{displaymath}
x_{2,k}x_{4,k}\cdots x_{2n,k} = \begin{cases}
                                x_2x_4\cdots x_{2n}, & k \textrm{ even} \\ 
                                x_1x_3\cdots x_{2n-1}, & k \textrm{ odd} \\
                                \end{cases}
\end{displaymath}
As such, it suffices to verify that these identities hold if the $x_{j,k}$ are given by \eqref{Tkx}.  

If $j+k$ is even then
\begin{align*}
x_{j,k} &= x_{j-3k}\left(\prod_{i=-k}^{k-1} y_{j+1+3i}\right)\frac{F_{j+2,k-1}F_{j-3,k}}{F_{j-2,k-1}F_{j+1,k}} \\
x_{j+1,k} &= x_{j+3k+1}\left(\prod_{i=-k}^{k-1} y_{j+2+3i}\right)\frac{F_{j-2,k-1}F_{j+3,k}}{F_{j+2,k-1}F_{j-1,k}}.
\end{align*}
Therefore,
\begin{align*}
x_{j,k}x_{j+1,k} &= x_{j-3k}\left(\prod_{i=-k}^{k-1} y_{j+1+3i}y_{j+2+3i}\right)x_{j+3k+1} \frac{F_{j-3,k}F_{j+3,k}}{F_{j-1,k}F_{j+1,k}} \\
                 &= x_{j-3k}\left(\prod_{i=-3k}^{3k} y_{j+i}\right)x_{j+3k+1} \frac{F_{j-3,k}F_{j+3,k}}{M_{j,k}F_{j-1,k}F_{j+1,k}}.
\end{align*}
But 
\begin{align*}
\prod_{i=-3k}^{3k} y_{j+i} &= \frac{(-x_{j-3k+1}x_{j-3k+2})\cdots (-x_{j+3k-1}x_{j+3k})} {(-x_{j-3k}x_{j-3k+1})(-x_{j-3k+2}x_{j-3k+3})\cdots(-x_{j+3k}x_{j+3k+1})} \\
                           &= -\frac{1}{x_{j-3k}x_{j+3k+1}}
\end{align*}
by \eqref{yvertexx}--\eqref{yedgex}.  Therefore,
\begin{displaymath}
x_{j,k}x_{j+1,k} = -\frac{F_{j-3,k}F_{j+3,k}}{M_{j,k}F_{j-1,k}F_{j+1,k}} = -y_{j,k}^{-1}
\end{displaymath}
by \eqref{Tk}.  A similar calculation shows $x_{j,k}x_{j+1,k} = -y_{j,k}$ for $j+k$ odd.

Finally, in computing $x_{2,k}x_{4,k}\cdots x_{2n,k}$, all of the $F$-polynomials in \eqref{Tkx} cancel out.  Each of $y_1,\ldots, y_{2n}$ appear exactly $k$ times in the product, but $y_1\cdots y_{2n} = 1$ so the $y$-variables do not contribute either.  All that remain are the $x_{j-3k}$ or $x_{j+3k}$ as appropriate.  So the product equals $x_2x_4\cdots x_{2n}$ if $k$ is even or $x_1x_3\cdots x_{2n-1}$ if $k$ is odd.
\end{proof}

It will be convenient in the following sections to define $M_{j,k}$ and $F_{j,k}$ for all $j,k$ (as opposed to just for $j+k$ even or, respectively, odd).  More specifically, let $M_{j,k} = \prod_{i=-k}^k y_{j+3i}$ for all $j$ and $k$, $F_{j,-1}=F_{j,0}=1$ for all $j$, and 
\begin{equation} \label{Fjkrec}
F_{j,k+1} = \frac{F_{j-3,k}F_{j+3,k} + (\prod_{i=-k}^k y_{j+3i})F_{j-1,k}F_{j+1,k}}{F_{j,k-1}} 
\end{equation}
for all $j$ and $k$ with $k \geq 0$.  

\section{Alternating sign matrix background}
An \emph{alternating sign matrix} is a square matrix of 1's, 0's, and -1's such that
\begin{itemize}
\item the non-zero entries of each row and column alternate in sign and
\item the sum of the entries of each row and column is 1.
\end{itemize}
Let $ASM(k)$ denote the set of $k$ by $k$ alternating sign matrices.  Alternating sign matrices are related to many other mathematical objects, including the posets $P_k$ used in the formula for the $F$-polynomials.  In this section, we explain the connection between these objects.

Recall that $Q_k$ is the set of triples $(r,s,t)\in \mathbb{Z}^3$ such that
\begin{displaymath}
2|s|-(k-2) \leq t \leq  k-2-2|r|
\end{displaymath}
and
\begin{displaymath}
2|s|-(k-2) \equiv t \equiv  k-2-2|r| \pmod{4}.
\end{displaymath}
Note that the first condition implies $|r|+|s| \leq k-2$, and the second implies $s+r \equiv k \pmod{2}$.  For each pair $(r,s)$ satisfying these properties, there are $\frac{1}{2}(k-|r|-|s|)$ points $(r,s,t) \in Q_k$, each separated by 4 units in the $t$-direction.  

Note that $Q_k$ and $Q_{k+1}$ are disjoint.  Let $P_k = Q_{k+1} \cup Q_k$.  Define a partial order on $P_k$ by saying that $(r',s',t')$ covers $(r,s,t)$  if and only if $t' = t + 1$ and $|r'-r| + |s'-s| = 1$.  The partial order on $P_k$ restricts to a partial order on $Q_k$.  

A bijection is given by Elkies, Kuperberg, Larsen, and Propp in \cite{EKLP} between $ASM(k)$ and $J(Q_k)$, the set of order ideals of $Q_k$.  This bijection is defined in several steps.  Given an order ideal $I$ of $Q_k$, associate to $I$ the height function $H$ defined by
\begin{displaymath}
H(r,s) = \begin{cases}
k+2 + \min\{t : (r,s,t) \in I\}, & \textrm{ if such a $t$ exists} \\
2|s|,                            & \textrm{ otherwise}
\end{cases}
\end{displaymath}
From $H$ construct a matrix $A^*$ with entries
\begin{displaymath}
a_{ij}^* = \frac{1}{2}H(-k+i+j,-i+j).
\end{displaymath}
Finally, the alternating sign matrix $A$ corresponding to $I$ is defined to be the matrix with entries
\begin{displaymath}
a_{ij} = \frac{1}{2}(a_{i-1,j}^* + a_{i,j-1}^* - a_{i-1,j-1}^* - a_{i,j}^*).
\end{displaymath}

As an example, the poset $Q_3$ (see Figure \ref{figureQ3}) has seven order ideals.  Table \ref{tableBijection} illustrates a couple instances of the bijection of $J(Q_3)$ with $ASM(3)$.

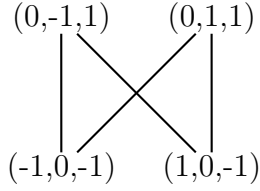
\begin{figure}
\begin{pspicture}(5,3)
\rput(1,2.5){\rnode{v1}{(0,-1,1)}}
\rput(3,2.5){\rnode{v2}{(0,1,1)}}
\rput(1,0.5){\rnode{v4}{(-1,0,-1)}}
\rput(3,0.5){\rnode{v5}{(1,0,-1)}}

\ncline{v4}{v1}
\ncline{v4}{v2}
\ncline{v5}{v1}
\ncline{v5}{v2}
\end{pspicture}
\caption{The poset $Q_3$}
\label{figureQ3}
\end{figure}

\begin{table}
\begin{tabular}{|c|c|c|}
\hline
$I$ & $\emptyset$ & $\{(-1,0,-1),(1,0,-1)\}$ \\
\hline
$H(r,s)$ &
$\begin{array}{ccccccc}
            &  &  & 6 &  &  & \\
            &  & 4 &  & 4 &  & \\
            & 2 &  & 2 &  & 2 & \\
           0 &  & 0 &  & 0 &  & 0\\
            & 2 &  & 2 &  & 2 & \\
            &  & 4 &  & 4 &  & \\
            &  &  & 6 &  &  & \\
\end{array}$ &
$\begin{array}{ccccccc}
           &  &  & 6 &  &  & \\
           &  & 4 &  & 4 &  & \\
           & 2 &  & 2 &  & 2 & \\
          0 &  & 4 &  & 4 &  & 0\\
           & 2 &  & 2 &  & 2 & \\
           &  & 4 &  & 4 &  & \\
           &  &  & 6 &  &  & \\
\end{array}$\\
\hline
& & \\
$A^*$ & 
$\left[ \begin{array}{cccc}
            0 & 1 & 2 & 3\\
            1 & 0 & 1 & 2\\
            2 & 1 & 0 & 1\\
            3 & 2 & 1 & 0\\
\end{array}\right]$ & 
$\left[ \begin{array}{cccc}
            0 & 1 & 2 & 3\\
            1 & 2 & 1 & 2\\
            2 & 1 & 2 & 1\\
            3 & 2 & 1 & 0\\
\end{array}\right]$\\
& & \\
\hline
& & \\
$A$ &
$\left[ \begin{array}{ccc}
            1 & 0 & 0 \\
            0 & 1 & 0 \\
            0 & 0 & 1 \\
\end{array}\right]$  &
$\left[ \begin{array}{ccc}
            0 & 1 & 0 \\
            1 & -1 & 1 \\
            0 & 1 & 0 \\
\end{array}\right] $\\
& & \\
\hline
\end{tabular}
\caption{Two examples illustrating the bijection between $J(Q_3)$ and $ASM(3)$.  The values of $H(r,s)$ are given for $r+s$ odd and $|r|+|s| \leq 3$ with $r$ increasing from left to right and $s$ from bottom to top.}
\label{tableBijection}
\end{table}

\begin{rem} The intermediate objects $H$ and $A^*$ in the bijection are themselves of interest.  The function $H$ is the so-called height function of a domino tiling in \cite{EKLP}.  In that paper, the posets $P_k$ and $Q_k$ are shifted upward, eliminating the need to add $k+2$ in the definition of $H$.  The $(n+1) \times (n+1)$ matrix $A^*$ (with row and column index starting at 0) is called the skew summation of $A$.  
\end{rem}

Call order ideals $I \subseteq Q_{k+1}$ and $J \subseteq Q_{k}$ compatible if $I \cup J$ is an order ideal of $P_k = Q_{k+1} \cup Q_k$.  Call alternating sign matrices $A \in ASM(k+1)$ and $B \in ASM(k)$ compatible if they correspond under the above bijection to compatible order ideals.  

This compatibility condition was introduced by D. Robbins and H. Rumsey in their study of a class of recurrences which includes the octahedron recurrence \cite{RR}.  A three-dimensional array of quantities $f_{i,j,k}$ is said to satisfy the \emph{octahedron recurrence} if
\begin{displaymath}
f_{i,j,k-1}f_{i,j,k+1} = f_{i-1,j,k}f_{i+1,j,k} + f_{i,j-1,k}f_{i,j+1,k}
\end{displaymath}
for all $(i,j,k) \in \mathbb{Z}^3$.  

Let $k \geq 1$ and consider the expression for $f_{0,0,k}$ in terms of the $(f_{i,j,-1})$ and $(f_{i,j,0})$.  It is easy to check that $f_{0,0,k}$ only depends on 
\begin{displaymath}
(f_{i,j,0} : |i|+|j| \leq k, i+j \equiv k \pmod{2})
\end{displaymath}
and
\begin{displaymath}
(f_{i,j,-1} : |i|+|j| \leq k-1, i+j-1 \equiv k \pmod{2}).
\end{displaymath}
Rotating by 45 degrees, the relevant initial values can be stored in the matrices
\begin{displaymath}
X_{k+1} = \left[ \begin{array}{cccc}
            f_{-k,0,0} & f_{-k+1,1,0} & \cdots & f_{0,k,0} \\
            f_{-k+1,-1,0} & f_{-k+2,0,0} & \cdots & f_{1,k-1,0} \\
            \vdots & \vdots & & \vdots \\
            f_{0,-k,0} & f_{1,-k+1,0} & \cdots & f_{k,0,0} \\
\end{array}\right]
\end{displaymath}
and
\begin{displaymath}
Y_k = \left[ \begin{array}{cccc}
            f_{-k+1,0,-1} & f_{-k+2,1,-1} & \cdots & f_{0,k-1,-1} \\
            f_{-k+2,-1,-1} & f_{-k+3,0,-1} & \cdots & f_{1,k-2,-1} \\
            \vdots & \vdots & & \vdots \\
            f_{0,-k+1,-1} & f_{1,-k+2,-1} & \cdots & f_{k-1,0,-1} \\
\end{array}\right].
\end{displaymath}

In the following, the notation $X^A$, with $X$ and $A$ matrices of the same dimensions, represents the product $\displaystyle\prod_i \displaystyle\prod_j x_{ij}^{a_{ij}}$.

\begin{prop}[{\cite[Theorem~ 1]{RR}}] \label{propOctASM}
Suppose $(f_{i,j,k})$ is a solution to the octahedron recurrence and let $k \geq 1$.  Then
\begin{displaymath}
f_{0,0,k} = \sum_{A, B}(X_{k+1})^A(Y_k)^{-B}
\end{displaymath}
where the sum is over all compatible pairs $A \in ASM(k+1)$, $B \in ASM(k)$.
\end{prop}

\section{Computation of the $F$-polynomials}
This section proves the formula for the $F$-polynomials given in \eqref{Fjk}.  

Define Laurent monomials $m_{i,j,k}$ for $k \geq -1$ recursively as follows.  Let
\begin{equation} \label{mij0}
m_{i,j,0} = \prod_{l=0}^{j-1}\prod_{m=0}^ly_{3i+j-4l+6m-1}
\end{equation}
and $m_{i,j,-1} = 1/m_{i,j,0}$ for all $i,j \in \mathbb{Z}$.  For $k \geq 1$, put
\begin{equation} \label{mijk}
m_{i,j,k} = \frac{m_{i-1,j,k-1}m_{i+1,j,k-1}}{m_{i,j,k-2}}.
\end{equation}

Note that in \eqref{mij0}, if $j \leq 0$ the conventions for products mentioned in the introduction are needed.  Applying these conventions and simplifying yields $m_{i,-1,0} = m_{i,0,0} = 1$ and
\begin{displaymath}
m_{i,j,0} = \prod_{l=j}^{-2}\prod_{m=l+1}^{-1}y_{3i+j-4l+6m-1} 
\end{displaymath}
for $j \leq -2$.  A portion of the array $m_{i,j,0}$ is given in Figure \ref{figuremij0}.

\begin{figure}
\begin{displaymath}
\begin{array}{rccccc}
& & \vdots & \vdots & \vdots & \\
(j=3) & \cdots & y_{-9}y_{-5}y_{-3}y_{-1}y_1y_3 & y_{-6}y_{-2}y_0y_2y_4y_6 & y_{-3}y_1y_3y_5y_7y_9  & \cdots \\
(j=2) & \cdots & y_{-6}y_{-2}y_0 & y_{-3}y_1y_3 & y_0y_4y_6  & \cdots \\
(j=1) & \cdots & y_{-3} & y_0 & y_3 & \cdots \\
(j=0) & \cdots & 1 & 1 & 1 & \cdots \\
(j=-1) & \cdots & 1 & 1 & 1 & \cdots \\
(j=-2) & \cdots & y_{-4} & y_{-1} & y_2 & \cdots \\
(j=-3) & \cdots & y_{-7}y_{-5}y_{-1} & y_{-4}y_{-2}y_2 & y_{-1}y_1y_5  & \cdots \\
& & \vdots & \vdots & \vdots & \\
& & (i=-1) & (i=0) & (i=1) &
\end{array}
\end{displaymath}
\caption{The monomials $m_{i,j,0}$} \label{figuremij0}
\end{figure}

\begin{prop} \label{propMFOct}
Let $f_{i,j,k} = m_{i,j,k}F_{3i+j,k}$ for all $i,j,k$ with $k \geq -1$.  Then $(f_{i,j,k})$ is a solution to the octahedron recurrence.
\end{prop}

\begin{proof}
Fix $i,j,k$ with $k \geq 0$.  Then
\begin{displaymath}
f_{i,j,k-1}f_{i,j,k+1} = (m_{i,j,k-1}m_{i,j,k+1})(F_{3i+j,k-1}F_{3i+j,k+1}).
\end{displaymath}
By \eqref{mijk} and \eqref{Fjkrec} we have
\begin{align*}
m_{i,j,k-1}m_{i,j,k+1} &=  m_{i-1,j,k}m_{i+1,j,k} \\
F_{3i+j,k-1}F_{3i+j,k+1} &=  F_{3i+j-3,k}F_{3i+j+3,k} + (M_{3i+j,k})F_{3i+j-1,k}F_{3i+j+1,k}. 
\end{align*}
Multiplying yields
\begin{equation} \label{MFOct}
f_{i,j,k-1}f_{i,j,k+1} = f_{i-1,j,k}f_{i+1,j,k} + (m_{i-1,j,k}m_{i+1,j,k}M_{3i+j,k})F_{3i+j-1,k}F_{3i+j+1,k}.
\end{equation}

\begin{lem} \label{lemmRec}
For all $i,j,k$ with $k \geq -1$,
\begin{displaymath}
\frac{m_{i,j-1,k}m_{i,j+1,k}}{m_{i-1,j,k}m_{i+1,j,k}} = M_{3i+j,k}
\end{displaymath}
for $M$ as defined in Section~ 4.
\end{lem}

\begin{proof}
Let $a_{i,j,k} = \frac{m_{i,j-1,k}m_{i,j+1,k}}{m_{i-1,j,k}m_{i+1,j,k}}$.  Applying the recurrence \eqref{mijk} defining the $m_{i,j,k}$ to each of the four monomials on the right hand side shows that the $a_{i,j,k}$ satisfy the same recurrence:
\begin{displaymath}
a_{i,j,k} = \frac{a_{i-1,j,k-1}a_{i+1,j,k-1}}{a_{i,j,k-2}}
\end{displaymath}
We want to show $a_{i,j,k} = M_{3i+j,k}$.  The fact that
\begin{displaymath}
M_{3i+j,k} = \frac{M_{3i+j-3,k-1}M_{3i+j+3,k-1}}{M_{3i+j,k-2}}
\end{displaymath}
was demonstrated in the proof of Proposition~ \ref{propMjk}.  All that remains is to check the initial conditions $k=0$ and $k=-1$. 

From \eqref{mij0} it is easy to show that $m_{i,j+1,0}/m_{i-1,j,0} = \prod_{l=0}^jy_{3i+j+2l}$.  Shifting yields $m_{i+1,j,0}/m_{i,j-1,0} = \prod_{l=0}^{j-1}y_{3i+j+2l+2}$.  Therefore,
\begin{displaymath}
a_{i,j,0} = \frac{m_{i,j+1,0}/m_{i-1,j,0}}{m_{i+1,j,0}/m_{i,j-1,0}} 
          = \frac{\prod_{l=0}^jy_{3i+j+2l}}{\prod_{l=0}^{j-1}y_{3i+j+2l+2}} 
          = y_{3i+j} 
          = M_{3i+j,0}.
\end{displaymath}
Taking the reciprocal of both sides yields 
\begin{displaymath}
a_{i,j,-1} = 1/a_{i,j,0} = 1/y_{3i+j} = M_{3i+j,-1}.
\end{displaymath}
So, $a_{i,j,k} = M_{3i+j,k}$ for all $k \geq -1$, as desired.
\end{proof}
By Lemma~ \ref{lemmRec}, $m_{i-1,j,k}m_{i+1,j,k}M_{3i+j,k} = m_{i,j-1,k}m_{i,j+1,k}$.  Substituting this into \eqref{MFOct} yields
\begin{align*}
f_{i,j,k-1}f_{i,j,k+1} &= f_{i-1,j,k}f_{i+1,j,k} + m_{i,j-1,k}m_{i,j+1,k}F_{3i+j-1,k}F_{3i+j+1,k} \\
                       &= f_{i-1,j,k}f_{i+1,j,k} + f_{i,j-1,k}f_{i,j+1,k}
\end{align*}
which completes the proof of Proposition~ \ref{propMFOct}.
\end{proof}

The general solution to the octahedron recurrence given in Proposition \ref{propOctASM} in particular applies to this solution. Now $m_{i,0,k} = 1$ for all $i$ and $k$, so $f_{0,0,k} = m_{0,0,k}F_{0,k} = F_{0,k}$.  For all $i$ and $j$, we have $f_{i,j,0} = m_{i,j,0}$ (since $F_{3i+j,0} = 1$), so the matrix $X_{k+1}$ is given by
\begin{equation} \label{Xm}
X_{k+1} = \left[ \begin{array}{cccc}
            m_{-k,0,0} & m_{-k+1,1,0} & \cdots & m_{0,k,0} \\
            m_{-k+1,-1,0} & m_{-k+2,0,0} & \cdots & m_{1,k-1,0} \\
            \vdots & \vdots & & \vdots \\
            m_{0,-k,0} & m_{1,-k+1,0} & \cdots & m_{k,0,0} \\
\end{array}\right].
\end{equation}
Finally $f_{i,j,-1} = m_{i,j,-1} = 1/m_{i,j,0}$ for all $i$ and $j$.  It follows that each entry of $Y_{k}$ is the reciprocal of the corresponding entry of $X_{k}$.  So, $(Y_k)^{-B} = X_k^B$ for any $B \in ASM(k-1)$.  This proves the following result.

\begin{prop} \label{propFASM}
Using the $m_{i,j,0}$ as the entries of $X_{k+1}$ and $X_k$ as in \eqref{Xm}
\begin{equation} \label{FASM}
F_{0,k} = \sum_{A,B}(X_{k+1})^A(X_k)^{B}
\end{equation}
where the sum is over all compatible pairs $A \in ASM(k+1)$, $B \in ASM(k)$.
\end{prop}

As an example, let $k=2$.  Then
\begin{displaymath}
X_2 = \left[ \begin{array}{cc}
m_{-1,0,0} & m_{0,1,0} \\
m_{0,-1,0} & m_{1,0,0} \\
\end{array}\right]
= \left[ \begin{array}{cc}
1 & y_0 \\
1 & 1 \\
\end{array}\right]
\end{displaymath}
and
\begin{displaymath}
X_3 = \left[ \begin{array}{ccc}
            m_{-2,0,0} & m_{-1,1,0}  & m_{0,2,0} \\
            m_{-1,-1,0} & m_{0,0,0}  & m_{1,1,0} \\
            m_{0,-2,0} & m_{1,-1,0} &  m_{2,0,0} \\
\end{array}\right]
= \left[ \begin{array}{ccc}
            1      & y_{-3} & y_{-3}y_1y_3 \\
            1      & 1      & y_3 \\
            y_{-1} & 1      & 1 \\
\end{array}\right].
\end{displaymath}
There are two elements $B \in ASM(2)$, namely the two permutation matrices, so $X_2^B$ is either 1 or $y_0$.  The seven alternating sign matrices $A \in ASM(3)$ are listed in Table \ref{tableASM3}.  The associated monomials $X_3^A$ are $1$,$y_{-3}$,$y_3$ (top row), $y_{-3}y_3$ (middle row), and $y_{-3}y_3y_{-1}, y_{-3}y_3y_1, y_{-3}y_3y_{-1}y_1$ (bottom row).  

\begin{table}
\begin{displaymath}
\begin{array}{ccc}
\hline \\
\left[ \begin{array}{ccc}
            1 & 0 & 0 \\
            0 & 1 & 0 \\
            0 & 0 & 1 \\
\end{array}\right] 
&
\left[ \begin{array}{ccc}
            0 & 1 & 0 \\
            1 & 0 & 0 \\
            0 & 0 & 1 \\
\end{array}\right] &
\left[ \begin{array}{ccc}
            1 & 0 & 0 \\
            0 & 0 & 1 \\
            0 & 1 & 0 \\
\end{array}\right] \smallskip \\
\hline \\
&
\left[ \begin{array}{ccc}
            0 & 1 & 0 \\
            1 & -1 & 1 \\
            0 & 1  & 0 \\
\end{array}\right]& \smallskip \\
\hline \\
\left[ \begin{array}{ccc}
            0 & 1 & 0 \\
            0 & 0 & 1 \\
            1 & 0 & 0 \\
\end{array}\right]&
\left[ \begin{array}{ccc}
            0 & 0 & 1 \\
            1 & 0 & 0 \\
            0 & 1 & 0 \\
\end{array}\right]&
\left[ \begin{array}{ccc}
            0 & 0 & 1 \\
            0 & 1 & 0 \\
            1 & 0 & 0 \\
\end{array}\right] \smallskip \\
\hline
\end{array}
\end{displaymath}
\caption{The seven elements of $ASM(3)$} \label{tableASM3}
\end{table}

The compatibility condition for $ASM(3)$ and $ASM(2)$ is as follows.  The three matrices in the top row of Table \ref{tableASM3} are related only to the identity matrix.  The three matrices in the bottom row are related only to the other element of $ASM(2)$.  Lastly, the middle matrix is related to both elements of $ASM(2)$.  The resulting formula is
\begin{align*}
F_{0,2} &= (1 + y_{-3} + y_{3} + y_{-3}y_3)1 \\
        &\quad + (y_{-3}y_3 + y_{-3}y_3y_{-1} + y_{-3}y_3y_1 +y_{-3}y_3y_{-1}y_1)y_0 \\
        &= (1+y_{-3})(1+y_3) + y_{-3}y_0y_3(1+y_{-1})(1+y_1)
\end{align*}
which matches \eqref{Fj2}.

Using the bijection between alternating sign matrices and order ideals from the previous section, \eqref{FASM} can be expressed in terms of order ideals.  Associate to each triple $(r,s,t) \in \mathbb{Z}^3$ the variable $y_{3r+s}$.  Define the weight of a finite subset $I$ of $\mathbb{Z}^3$ (particularly an order ideal of $Q_k$ or $P_k$) to be the monomial
\begin{displaymath}
\wt(I) = \prod_{(r,s,t) \in I}y_{3r+s}.
\end{displaymath}

\begin{prop} \label{propASMIdeals}
If $I \subseteq Q_k$ is an order ideal and $A \in ASM(k)$ is the associated alternating sign matrix then
\begin{displaymath} 
\wt(I) = X_k^A.
\end{displaymath}
\end{prop}
\begin{proof}
The proof is by induction on the number of elements of $I$.  If $I$ is empty then $\wt(I) = 1$.  On the other hand, the height function corresponding to $I$ is $H(r,s) = 2|s|$ so the matrix $A^*$ has entries $a_{ij}^* = |-i+j|$.  It follows that $a_{ij} =1$ if $i=j$ and $a_{ij} = 0$ otherwise (i.e. $A$ is the identity matrix).  The diagonal entries of $X_k$  are all 1, so $X_k^A=1$.

Now suppose $I$ is non-empty and that the proposition holds for all smaller order ideals.  Then there is some $(r_0,s_0,t_0) \in I$ such that $I' = I - \{(r_0,s_0,t_0)\}$ is still an order ideal of $Q_k$.  By the induction hypothesis, $\wt(I') = X_k^{A'}$ where $A'$ is the alternating sign matrix corresponding to $I'$.  Clearly, $\wt(I) = y_{3r_0+s_0}\wt(I')$.  The addition of $(r_0,s_0,t_0)$ to obtain $I$ from $I'$ propagates through the bijection as follows:
\begin{align*}
I &= I' \cup \{(r_0,s_0,t_0)\} \\
H(r,s) &= \begin{cases}
H'(r,s) + 4, & (r,s) = (r_0,s_0) \\
H'(r,s),     & \textrm{ otherwise}
\end{cases}\\
a_{ij}^* &= \begin{cases}
a_{ij}^{*'} + 2, & (i,j)=(i_0,j_0) \\
a_{ij}^{*'},     & \textrm{ otherwise}
\end{cases}\\
a_{ij} &= \begin{cases}
a_{ij}' - 1, & (i,j) \in \{(i_0,j_0),(i_0+1,j_0+1)\} \\
a_{ij}' + 1, & (i,j) \in \{(i_0+1,j_0),(i_0,j_0+1)\} \\
a_{ij}',     & \textrm{ otherwise}
\end{cases}
\end{align*}
where $i_0,j_0$ are the integers satisfying $-k+i_0+j_0=r_0$, $-i_0+j_0 = s_0$.

The $(i_0,j_0)$ entry of $X_k$ is $m_{-k-1+i_0+j_0,j_0-i_0,0} = m_{r_0-1,s_0,0}$.  Similarly, the $(i_0,j_0+1)$ entry is $m_{r_0,s_0+1,0}$, the $(i_0+1,j_0)$ entry is $m_{r_0,s_0-1,0}$, and the $(i_0+1,j_0+1)$ entry is $m_{r_0+1,s_0,0}$.  These are the only entries where $A$ and $A'$ differ, so
\begin{align*}
X_k^A &= \frac{m_{r_0,s_0+1,0}m_{r_0,s_0-1,0}}{m_{r_0-1,s_0,0}m_{r_0+1,s_0,0}}X_k^{A'} \\
      &= y_{3r_0+s_0}X_k^{A'} 
\end{align*}
by Lemma~ \ref{lemmRec}.  Therefore $X_k^A = y_{3r_0+s_0}\wt(I') = \wt(I)$ as desired.
\end{proof}

\begin{ex}
Let $A$ be the matrix in the middle row of Table \ref{tableASM3}.  We have already seen that $X_3^A = y_{-3}y_3$.  According to Table \ref{tableBijection}, the corresponding order ideal is $I=\{(-1,0,-1),(1,0,-1)\}$, so $\wt(I)=y_{-3}y_3$ as well.
\end{ex}

\begin{thm}
\begin{displaymath}
F_{j,k} = \sum_{I \in J(P_k)} \prod_{(r,s,t)\in I} y_{3r+s+j}
\end{displaymath}
where $J(P_k)$ denotes the set of order ideals of $P_k$.
\end{thm}
\begin{proof}
The effect on either side of this equation of changing $j$ is to shift the index of each $y$-variable.  As such, it suffices to verify the formula for $F_{0,k}$.  By Proposition~ \ref{propFASM},
\begin{displaymath}
F_{0,k} = \sum_{A,B} X_{k+1}^A X_k^B
\end{displaymath}
where the sum is over compatible pairs $A \in ASM(k+1), B \in ASM(k)$.  Such pairs are in bijection with compatible pairs of order ideals $I_1 \subseteq Q_{k+1}, I_2 \subseteq Q_k$, so by Proposition~ \ref{propASMIdeals}
\begin{displaymath}
F_{0,k} = \sum_{I_1,I_2} \wt(I_1)\wt(I_2).
\end{displaymath}
Every order ideal $I$ of $P_k$ is uniquely the union of such a compatible pair, namely $I_1 = I \cap Q_{k+1}$ and $I_2 = I \cap Q_k$.  Since $I$ is a disjoint union of $I_1$ and $I_2$, it follows that $\wt(I) = \wt(I_1)\wt(I_2)$.  Therefore,
\begin{displaymath}
F_{0,k} = \sum_{I \in J(P_k)} \wt(I) 
        = \sum_{I \in J(P_k)} \prod_{(r,s,t)\in I} y_{3r+s}. \qedhere
\end{displaymath}
\end{proof}

\section{Axis-aligned polygons}
The polynomial $F_{0,k}$ takes an interesting form under the specialization $y_j=-1$ for all $j \equiv k \pmod{2}$.  Specifically, there is a $(k+1) \times (k+1)$ matrix whose entries are monomials in the other $y_j$ and whose determinant equals $F_{0,k}$ in this case.  This specialization is of geometric interest because it arises for axis-aligned polygons.

Let $(f_{i,j,k})$ be a solution to the octahedron recurrence and let $\tilde{f}_{i,j,k} = \sigma_j f_{i,j,k}$ where
\begin{displaymath}
\sigma_j = \begin{cases}
1, &  j \equiv 0,3 \pmod{4} \\
-1, & j \equiv 1,2 \pmod{4} 
\end{cases}
\end{displaymath}
The $\tilde{f}_{i,j,k}$ satisfy a slightly different recurrence.  Indeed
\begin{align*}
\tilde{f}_{i,j,k+1}\tilde{f}_{i,j,k-1} &= f_{i,j,k+1}f_{i,j,k-1} \\
                                       &= f_{i-1,j,k}f_{i+1,j,k} + f_{i,j-1,k}f_{i,j+1,k} \\
                                       &= \tilde{f}_{i-1,j,k}\tilde{f}_{i+1,j,k} - \tilde{f}_{i,j-1,k}\tilde{f}_{i,j+1,k}
\end{align*}
since $\sigma_j^2 = 1$ and $\sigma_{j-1}\sigma_{j+1} = -1$ for all $j$.  

This recurrence is the one used in Dodgson's method of computing determinants.  More specifically, recall that $f_{0,0,k}$ can be expressed in terms of the initial conditions given in the matrices $X_{k+1}$ and $Y_k$.  Similarly, $\tilde{f}_{0,0,k}$ can be expressed in terms of matrices $\tilde{X}_{k+1}$ and $\tilde{Y}_k$ in which the $f$'s have been replaced by $\tilde{f}$'s.  According to Dodgson's condensation, if $\tilde{Y}_k$ has all its entries equal to 1 then $\tilde{f}_{0,0,k} = \det(\tilde{X}_{k+1}$).

\begin{prop} \label{propF0k}
If $y_j = -1$ for all $j \equiv k \pmod{2}$ then
\begin{displaymath}
F_{0,k} = \det \left[ \begin{array}{cccc}
            \tilde{m}_{-k,0,0} & \tilde{m}_{-k+1,1,0} & \cdots & \tilde{m}_{0,k,0} \\
            \tilde{m}_{-k+1,-1,0} & \tilde{m}_{-k+2,0,0} & \cdots & \tilde{m}_{1,k-1,0} \\
            \vdots & \vdots & & \vdots \\
            \tilde{m}_{0,-k,0} & \tilde{m}_{1,-k+1,0} & \cdots & \tilde{m}_{k,0,0} \\
\end{array}\right]
\end{displaymath}
where $\tilde{m}_{i,j,0} = \sigma_jm_{i,j,0}$.
\end{prop}

\begin{proof}
Consider the solution to the octahedron recurrence constructed in Section~ 6.  The entries of $Y_k$ are of the form $f_{i,j,-1}$ for $i + j - 1\equiv k \pmod{2}$.  Recall that $f_{i,j,-1} = m_{i,j,-1} = 1/m_{i,j,0}$.  However, by \eqref{mij0} we have that $m_{i,j,0}$ is equal to a product of $-1$'s since $3i+j-4l+6m-1 \equiv k \pmod{2}$.  The number of terms in this product equals $\sum_{l=0}^{j-1} (l+1) = j(j+1)/2$.  Therefore, $m_{i,j,0} = \sigma_j$, so $m_{i,j,-1} = \sigma_j$.  All of the entries of $Y_k$ equal $\sigma_j$, so all the entries of $\tilde{Y}_k$ equal $\sigma_j^2 = 1$.

It follows from Dodgson's condensation that $\tilde{f}_{0,0,k} = \det(\tilde{X}_{k+1}$).  The matrix $\tilde{X}_{k+1}$ is exactly the one in the statement of this proposition, and $\tilde{f}_{0,0,k} = f_{0,0,k} = F_{0,k}$.
\end{proof}

For $k=1$, the proposition says 
\begin{displaymath}
F_{0,1} = \det \left[ \begin{array}{cc}
            m_{-1,0,0} & -m_{0,1,0} \\
            m_{0,-1,0} & m_{1,0,0} \\
\end{array}\right]
= \det \left[ \begin{array}{cc}
            1 & -y_0 \\
            1 & 1 \\
\end{array}\right] 
\end{displaymath}
which equals $1 + y_0$.  For $k=2$,
\begin{align*}
F_{0,2} &= \det \left[ \begin{array}{ccc}
            m_{-2,0,0} & -m_{-1,1,0} & -m_{0,2,0} \\
            m_{-1,-1,0} & m_{0,0,0} & -m_{1,1,0} \\
            -m_{0,-2,0} & m_{1,-1,0} & m_{2,0,0} \\
\end{array}\right] \\
&= \det \left[ \begin{array}{ccc}
            1 & -y_{-3} & -y_{-3}y_1y_3 \\
            1 & 1 & -y_3 \\
            -y_{-1} & 1 & 1 \\
\end{array}\right].
\end{align*}
This determinant agrees with the result of substituting $y_0 = -1$ into \eqref{Fj2}.

The remainder of this paper is devoted to axis-aligned polygons, i.e. polygons whose sides are alternately parallel to the $x$ and $y$ axes.  Note that such a polygon has the property that its even edges are concurrent (since they pass through a common point at infinity) and its odd edges are also concurrent.  We give a new proof of the result \cite{S} that the pentagram map takes an axis-aligned polygon to one which satisfies the dual condition (its even vertices are collinear and its odd vertices are also collinear) after a certain number of steps.

\begin{lem} \label{lemyDeg}
Let $A$ be a twisted polygon indexed either by $\mathbb{Z}$ or $\frac{1}{2} + \mathbb{Z}$.  Suppose that no 3 consecutive points of $A$ are collinear.  Then for each index $i$ of $A$:
\begin{enumerate}
\item $A_{i-2},A_i,A_{i+2}$ are collinear if and only if $y_{2i}(A) = -1$.
\item $\overleftrightarrow{A_{i-2}A_{i-1}},\overleftrightarrow{A_iA_{i+1}},\overleftrightarrow{A_{i+2}A_{i+3}}$ are concurrent if and only if $y_{2i+1}(A) = -1$.
\end{enumerate}
\end{lem}

\begin{proof}
By \eqref{yvertex}, we have $y_{2i}(A) = -1$ if and only if 
\begin{displaymath}
\chi (\overleftrightarrow{A_iA_{i-2}}, \overleftrightarrow{A_iA_{i-1}}, \overleftrightarrow{A_iA_{i+1}}, \overleftrightarrow{A_iA_{i+2}}) = 1.
\end{displaymath}
A cross ratio $\chi(a,b,c,d)$ equals 1 if and only if $a=d$ or $b=c$.  By assumption, $A_{i-1}, A_i, A_{i+1}$ are not collinear so  $\overleftrightarrow{A_iA_{i-1}} \neq \overleftrightarrow{A_iA_{i+1}}$.  Therefore $y_{2i}(A) = -1$, if and only if $\overleftrightarrow{A_iA_{i-2}} = \overleftrightarrow{A_iA_{i+2}}$, i.e. $A_{i-2},A_i,A_{i+2}$ are collinear.  The proof of the second statement is similar.
\end{proof}

Let $A \in \mathcal{P}_{2n}$ be an axis-aligned polygon.  Suppose in addition that $A$ is closed, i.e. $A_{i+2n} = A_i$ for all $i \in \mathbb{Z}$.  Let $s_{2j+1}$ denote the signed length of the side joining $A_j$ and $A_{j+1}$, where the sign is taken to be positive if and only if $A_{j+1}$ is to the right of or above $A_j$.  An example of an axis-aligned octagon is given in Figure \ref{figureOctogon}.  It follows from the second statement in Lemma~ \ref{lemyDeg} that $y_{2j+1}(A) = -1$ for all $j \in \mathbb{Z}$.  On the other hand, the even $y$-parameters can be expressed directly in terms of the side lengths.

\begin{figure} 
\begin{pspicture}(7,5)
  \pspolygon[showpoints=true](1,1)(4,1)(4,2)(6,2)(6,4)(3,4)(3,3)(1,3)
  \uput[dl](1,1){$A_1$}
  \uput[dr](4,1){$A_2$}
  \uput[dr](4,2){$A_3$}
  \uput[dr](6,2){$A_4$}
  \uput[ur](6,4){$A_5$}
  \uput[ul](3,4){$A_6$}
  \uput[ul](3,3){$A_7$}
  \uput[ul](1,3){$A_8$}
	\psset{labelsep=2pt}
  \uput[r](1,2){$s_1$}
  \uput[u](2.5,1){$s_3$}
  \uput[l](4,1.5){$s_5$}
  \uput[u](5,2){$s_7$}
  \uput[l](6,3){$s_9$}
  \uput[d](4.5,4){$s_{11}$}
  \uput[r](3,3.5){$s_{13}$}
  \uput[d](2,3){$s_{15}$}
  \psset{labelsep=5pt}
\end{pspicture}
\caption{An axis-aligned octagon.  The side lengths $s_3,s_5,s_7$, and $s_9$ are positive and the others are negative.} \label{figureOctogon}
\end{figure}
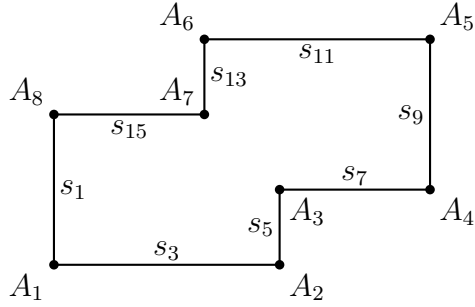

\begin{lem} \label{lemyEvenDeg}
For all $j \in \mathbb{Z}$ 
\begin{displaymath}
y_{2j}(A) = -\frac{s_{2j-1}s_{2j+1}}{s_{2j-3}s_{2j+3}}.
\end{displaymath}
\end{lem}

\begin{proof}
By \eqref{yvertex}
\begin{displaymath}
y_{2j}(A) = -\left(\chi (\overleftrightarrow{A_jA_{j-2}}, \overleftrightarrow{A_jA_{j-1}}, \overleftrightarrow{A_jA_{j+1}}, \overleftrightarrow{A_jA_{j+2}})\right)^{-1}.
\end{displaymath}
A cross ratio of 4 lines can be calculated as the cross ratio of the 4 corresponding slopes.  Suppose first that $\overleftrightarrow{A_jA_{j-1}}$ is vertical (slope = $\infty$) and $\overleftrightarrow{A_jA_{j+1}}$ is horizontal (slope = 0).  Then the slope of $\overleftrightarrow{A_jA_{j-2}}$ is $s_{2j-1}/s_{2j-3}$ and the slope of $\overleftrightarrow{A_jA_{j+2}}$ is $s_{2j+3}/s_{2j+1}$.  Therefore
\begin{align*}
y_{2j}(A) &= -\left(\chi\left(\frac{s_{2j-1}}{s_{2j-3}}, \infty,0, \frac{s_{2j+3}}{s_{2j+1}}\right)\right)^{-1} \\
          &= -\left(\frac{(-\infty)\left(-\frac{s_{2j+3}}{s_{2j+1}}\right)} {\left(\frac{s_{2j-1}}{s_{2j-3}}\right)(\infty)}\right)^{-1} \\
          &= -\frac{s_{2j-1}s_{2j+1}}{s_{2j-3}s_{2j+3}}
\end{align*}
as desired.  The calculation is similar if $\overleftrightarrow{A_jA_{j-1}}$ is horizontal and $\overleftrightarrow{A_jA_{j+1}}$ is vertical.
\end{proof}

\begin{prop} \label{propFeq0}
Let $A \in \mathcal{P}_{2n}$ be closed and axis-aligned, and let $y_j = y_j(A)$ for all ~ $j$.  If $n$ is even then $F_{0,n-1}(y) = 0$.
\end{prop}

\begin{proof}
By Lemma~ \ref{lemyDeg}, $y_j = -1$ for all odd $j$, that is for all $j \equiv n-1 \pmod{2}$.  So by Proposition~ \ref{propF0k} we have
\begin{displaymath}
F_{0,n-1} = \det \left[ \begin{array}{cccc}
            \tilde{m}_{-n+1,0,0} & \tilde{m}_{-n+2,1,0} & \cdots & \tilde{m}_{0,n-1,0} \\
            \tilde{m}_{-n+2,-1,0} & \tilde{m}_{-n+3,0,0} & \cdots & \tilde{m}_{1,n-2,0} \\
            \vdots & \vdots & & \vdots \\
            \tilde{m}_{0,-n+1,0} & \tilde{m}_{1,-n+2,0} & \cdots & \tilde{m}_{n-1,0,0} \\
         \end{array} \right].
\end{displaymath}
We want to show that this matrix, call it $X$, is degenerate.  Let $x_{ij}$ be the $i,j$ entry of $X$, that is $x_{ij} = \tilde{m}_{-n-1+i+j,j-i,0}$.  Then 
\begin{align*}
\frac{x_{i,j}x_{i+1,j+1}}{x_{i+1,j}x_{i,j+1}}
    &= \frac{\tilde{m}_{-n-1+i+j,j-i,0}\tilde{m}_{-n+1+i+j,j-i,0}}{\tilde{m}_{-n+i+j,j-i-1,0}\tilde{m}_{-n-1+i+j,j-i+1,0}} \\
    &= -\frac{m_{-n-1+i+j,j-i,0}m_{-n+1+i+j,j-i,0}}{m_{-n+i+j,j-i-1,0}m_{-n-1+i+j,j-i+1,0}} \\
    &= -\frac{1}{y_{3(-n+i+j) + j-i}} = -\frac{1}{y_{-3n + 2i + 4j}}
\end{align*}
by Lemma~ \ref{lemmRec}.
Since $n$ is even, $-3n + 2i + 4j$ is even so
\begin{align*}
\frac{x_{i,j}x_{i+1,j+1}}{x_{i+1,j}x_{i,j+1}} &= \frac{s_{-3n+2i+4j-3}s_{-3n+2i+4j+3}}{s_{-3n+2i+4j-1}s_{-3n+2i+4j+1}} \\
                                              &= \frac{z_{i,j}z_{i+1,j+1}}{z_{i+1,j}z_{i,j+1}}
\end{align*}
where $z_{ij} = s_{-3n+2i+4j-3}$.  Therefore
\begin{displaymath}
\frac{x_{i,j}}{z_{i,j}}\frac{x_{i+1,j+1}}{z_{i+1,j+1}} - \frac{x_{i+1,j}}{z_{i+1,j}}\frac{x_{i,j+1}}{z_{i,j+1}} = 0.
\end{displaymath}

The matrix with entries $x_{ij} / z_{ij}$ has consecutive $2\times 2$ minors equal to 0.  The entries of this matrix are generically non-zero, so it follows that all of its $2 \times 2$ minors vanish.  Therefore, the matrix has rank 1 and there exist non-zero scalars $\lambda_1,\ldots, \lambda_n$ and $\mu_1,\ldots, \mu_n$ such that $x_{ij}/z_{ij} = \lambda_i\mu_j$.

Now the matrix $Z$ whose entries are $z_{ij}$ is degenerate because its rows all have sum 0.  Indeed for each fixed $i$
\begin{displaymath}
\sum_{j=1}^n z_{ij} = \sum_{j=1}^n s_{-3n+2i+4j-3} = 0
\end{displaymath}
because $A$ is closed so the sum of the lengths of its $n$ horizontal (or vertical) sides must be 0.  However, $X$ can be obtained from $Z$ by multiplying its rows by the $\lambda_i$ and its columns by the $\mu_j$.  Therefore, $X$ is degenerate as well.
\end{proof}

\begin{cor}
Let $y_j = y_j(A)$ for $A$ as above.  Then $F_{j,n-1}(y) = 0$ for all $j \equiv n \pmod{2}$.
\end{cor}

\begin{proof}
Suppose first that $n$ is even.  Then by Proposition~ \ref{propFeq0} we have $F_{0,n-1} = 0$.  Cyclically permuting the vertex indexing has the effect of shifting the $y$-variables, and hence the $F$-polynomials, by an even offset.  So $F_{j,n-1} = 0$ for all even $j$.

Now suppose $n$ is odd.  Shifting all of the $y$-variables up by 1 in the statement of Proposition~ \ref{propF0k} yields that if $y_j = -1$ for all $j-1 \equiv k \pmod{2}$ then $F_{1,k}$ is given by the determinant of a matrix.  Since $A$ is axis-aligned, $y_j = -1$ for all odd $j$, that is for all $j-1 \equiv n-1 \pmod{2}$.  Therefore, $F_{1,n-1}$ is the determinant of some matrix $X$ which is exactly like the matrix $X$ in the proof of Proposition~ \ref{propFeq0}, except that the $y$-variables have all been shifted by 1.  The same proof shows that this matrix is degenerate so $F_{1,n-1}=0$.  Permuting the vertices yields that $F_{j,n-1} = 0$ for all odd ~ $j$. 
\end{proof}

\begin{thm}[Schwartz] \label{thmDeg}
Let $A$ be a closed, axis-aligned $2n$-gon.  Then the odd vertices of $T^{n-2}(A)$ are collinear, as are its even vertices.
\end{thm}

\begin{proof}
Suppose without loss of generality that $A$ is indexed by $\mathbb{Z}$.  We have already shown that for such a polygon, $F_{j,n-1} = 0$ provided $j \equiv n \pmod{2}$.  Therefore,
\begin{align*}
0 &= F_{j,n-1}F_{j,n-3} \\
  &= F_{j-3,n-2}F_{j+3,n-2} + \left(\prod_{i=-n+2}^{n-2} y_{j+3i}\right)F_{j-1,n-2}F_{j+1,n-2}.
\end{align*}
Hence
\begin{displaymath}
\left(\prod_{i=-n+2}^{n-2} y_{j+3i}\right)\frac{F_{j-1,n-2}F_{j+1,n-2}}{F_{j-3,n-2}F_{j+3,n-2}} = -1.
\end{displaymath}
By \eqref{Tk} the left hand side equals $y_{j,n-2}$, the $j$th $y$-parameter of $T^{n-2}(A)$.  So $y_{j,n-2} = -1$ for all $j \equiv n \pmod{2}$.

If $n$ is even, then $T^{n-2}(A)$ is indexed by $\mathbb{Z}$ and $y_{j,n-2} = -1$ for all $j$ even.  On the other hand, if $n$ is odd, then $T^{n-2}(A)$ is indexed by $\frac{1}{2} + \mathbb{Z}$ and $y_{j,n-2} = -1$ for all $j$ odd.  In either case, it follows from the first statement of Lemma~ \ref{lemyDeg} that the odd vertices of $T^{n-2}(A)$ lie on one line and the even vertices lie on another.
\end{proof}

\begin{rem}
Theorem~ \ref{thmDeg} is stated for all $n$ in \cite{S2} and proven for $n$ even (i.e. the number of sides of $A$ divisible by 4) in \cite{S}.  Schwartz's proof in \cite{S} also involves Dodgson's condensation, so it seems as though our proof must be related to his.  However, we are not sure what the connection is at this point.
\end{rem}

\begin{rem}
Theorem~ \ref{thmDeg} is only meant to hold for polygons $A$ for which the map $T^{n-2}$ is defined.  Additionally, the application of Lemma~ \ref{lemyDeg} at the end of the proof assumes that no 3 consecutive vertices of $T^{n-2}(A)$ are collinear.  The set of $A \in \mathcal{P}_{2n}$ satisfying these properties is open, but it could, a priori, be empty.  To rule out this possibility, it suffices to find a single example which works for each $n$.   According to Schwartz, there is substantial experimental evidence to suggest that this is always possible \cite{S2}. 
\end{rem}

Suppose now that $A$ is not closed, but twisted with $A_{i+2n} = \phi(A_i)$.  Since $A$ is axis-aligned, the projective transformation $\phi$ must send vertical lines to vertical lines and horizontal lines to horizontal lines.  One can check that all such projective transformations are of the form
\begin{displaymath}
\phi(x,y) = (ax+b,cy+d)
\end{displaymath}
for some reals $a,b,c,d$.

As before, let $s_{2i+1}$ be the signed length of the side joining $A_i$ to $A_{i+1}$.  Since $A$ is not closed, the side lengths are no longer periodic.  More specifically, if $s_j$ is the length of a horizontal edge then $s_{j+4n} = as_j$, while if it is the length of a vertical edge then $s_{j+4n} = cs_j$.  If we place the additional assumption that $a=c$, then $s_{j+4n}/s_j=a=c$ for all odd $j$.  The assumption that $a=c$ means that $\phi$ preserves the slopes of lines, or put another way, that it fixes every point at infinity.  Amazingly, under this assumption the result of Theorem~ \ref{thmDeg} still holds, except that $n-1$ applications of the pentagram map are needed instead of just $n-2$.

\begin{thm} \label{thmDegTwisted}
Let $A$ be a twisted, axis-aligned $2n$-gon with $A_{i+2n} = \phi(A_{i})$ and suppose that $\phi$ fixes every point at infinity.  Then the odd vertices of $T^{n-1}(A)$ are collinear, as are its even vertices.
\end{thm}

\begin{proof}
Following the proof of Theorem~ \ref{thmDeg}, it suffices to show that $F_{j,n} = 0$ for all $j \equiv n-1 \pmod{2}$.  By symmetry, it is enough to show that $F_{0,n} = 0$ if $n$ is odd and $F_{1,n} = 0$ if $n$ is even.  Suppose $n$ is odd.  Since $A$ is axis-aligned, Proposition~ \ref{propF0k} applies and
\begin{displaymath}
F_{0,n} = \det \left[ \begin{array}{cccc}
            \tilde{m}_{-n,0,0} & \tilde{m}_{-n+1,1,0} & \cdots & \tilde{m}_{0,n,0} \\
            \tilde{m}_{-n+1,-1,0} & \tilde{m}_{-n+2,0,0} & \cdots & \tilde{m}_{1,n-1,0} \\
            \vdots & \vdots & & \vdots \\
            \tilde{m}_{0,-n,0} & \tilde{m}_{1,-n+1,0} & \cdots & \tilde{m}_{n,0,0} \\
         \end{array} \right].
\end{displaymath}
Let $X$ be this matrix, that is, $x_{ij} = \tilde{m}_{-n-2+i+j,j-i,0}$.  As in the proof of Proposition~ \ref{propFeq0}, we have
\begin{displaymath}
\frac{x_{i,j}x_{i+1,j+1}}{x_{i+1,j}x_{i,j+1}} = -\frac{1}{y_{-3n + 2i + 4j-3}}
\end{displaymath}
for all $i$ and $j$.  Therefore,
\begin{displaymath}
\frac{x_{i,1}x_{i+1,n+1}}{x_{i+1,1}x_{i,n+1}} = -\frac{1}{y_{-3n+2i+1}y_{-3n+2i+5}\cdots y_{-3n+2i+4n-3}}.
\end{displaymath}

Expressing the $y$-parameters in terms of the side lengths using Lemma~ \ref{lemyEvenDeg}, the right hand side becomes a telescoping product leaving 
\begin{displaymath}
\frac{x_{i,1}x_{i+1,n+1}}{x_{i+1,1}x_{i,n+1}} = \left(\frac{s_{-3n+2i-2}}{s_{-3n+2i}}\right) \left(\frac{s_{n+2i}}{s_{n+2i-2}}\right) = 1
\end{displaymath}
since $s_{n+2i}/s_{-3n+2i} = s_{n+2i-2}/s_{-3n+2i-2}$ by the assumption on $\phi$.  Therefore $x_{i,1}/x_{i,n+1} = x_{i+1,1}/x_{i+1,n+1}$ for all $i=1,\ldots, n$.  The first and last columns of $X$ are linearly dependent, so $F_{0,n} = \det(X) = 0$ as desired.

The proof that $F_{1,n} = 0$ for $n$ even is similar.
\end{proof}

\begin{rem}
It should be possible to deduce Theorem~ \ref{thmDeg} from Theorem~ \ref{thmDegTwisted} along the following lines.  Let $A$ be a closed, axis-aligned $2n$-gon.  If the vertices of a polygon lie alternately on 2 lines, then the pentagram map collapses it to a single point (the intersection of those 2 lines).  As such, it suffices to show that $T^{n-1}(A)$ is a single point.  Approximating $A$ by twisted polygons with $\phi$ being smaller and smaller vertical translations, Theorem~ \ref{thmDegTwisted} shows that the vertices of $T^{n-1}(A)$ lie on 2 lines, $l_1$ and $l_2$.  In fact, it is easy to show that these lines must be parallel and, in this case, vertical.  Similarly, approximating $A$ by twisted polygons with smaller and smaller horizontal translations shows that the vertices of $T^{n-1}(A)$ lie on 2 horizontal lines $m_1$ and $m_2$.  Combining these, the vertices of $T^{n-1}(A)$ alternate between the points $l_1 \cap m_1$ and $l_2 \cap m_2$.  The pentagram map never collapses a polygon to a line segment of positive length, so it follows that these 2 points are equal.
\end{rem}

\end{document}